\documentclass[preprint,3p,times] {elsarticle}

\usepackage{amsmath,amsfonts,amsopn,amssymb}
\usepackage{graphicx}
\graphicspath{ {./Figures/} {./figure} }
\usepackage[caption=false]{subfig}
\usepackage{multirow,relsize,makecell}
\usepackage{url}

\usepackage{bm}

\usepackage{algorithm,algpseudocode}

\usepackage{hyperref}
\hypersetup{
	colorlinks=true,
	linkcolor=green,
	filecolor=magenta,      
	urlcolor=cyan,
	pdftitle={Overleaf Example},
	pdfpagemode=FullScreen,
}
\usepackage{cleveref} 

\usepackage{amsthm}
\newtheorem{theorem}{Theorem}
\newtheorem{lemma}{Lemma}
\newtheorem{proposition}{Proposition}
\newtheorem{remark}{Remark}

\numberwithin{theorem}{section}
\crefname{remark}{Remark}{Remarks}
\crefname{assumption}{Assumption}{Assumptions}
\crefname{algorithm}{Algorithm}{Algorithms} 
\crefname{theorem}{Theorem}{Theorems} 
\crefname{lemma}{Lemma}{Lemmas} 
\crefname{figure}{Figure}{Figures}

\begin{document}

\begin{frontmatter}



\title{Randomized Greedy Algorithms for Neural Network Optimization
}

\author[1]{Jinchao Xu}
\author[1]{Xiaofeng Xu
\footnote{Corresponding author: xiaofeng.xu@kaust.edu.sa}
}

\affiliation[1]{organization={Applied Mathematics and Computational Sciences Program, Computer, Electrical and Mathematical Science and Engineering Division, King Abdullah University of Science and Technology~(KAUST)},
            addressline={},
            city={Thuwal},
            postcode={23955},
            state={},
            country={Saudi Arabia}}

\begin{abstract}
Greedy algorithms have been successfully analyzed and applied in training neural networks for solving variational problems, ensuring guaranteed convergence orders.
In this paper, we extend the analysis of the orthogonal greedy algorithm (OGA) to convex optimization problems, establishing its optimal convergence rate.
 This result broadens the applicability of OGA by generalizing its optimal convergence rate from function approximation to convex optimization problems.
 In addition, we also address the issue regarding practical applicability of greedy algorithms, which is due to significant computational costs from the subproblems that involve an exhaustive search over a discrete dictionary. 
 We propose to use a more practical approach of randomly discretizing the dictionary at each iteration of the greedy algorithm. 
 We quantify the required size of the randomized discrete dictionary and prove that, with high probability, the proposed algorithm realizes a weak greedy algorithm, achieving optimal convergence orders.
Through numerous numerical experiments on function approximation, linear and nonlinear elliptic partial differential equations, we validate our analysis on the optimal convergence rate and demonstrate the advantage of using randomized discrete dictionaries over a deterministic one by showing orders of magnitude reductions in the size of the discrete dictionary, particularly in higher dimensions. 


\end{abstract}



\end{frontmatter}

\section{Introduction}
\label{sec:intro}
In recent years, neural networks have been widely used in scientific computing, particularly for solving partial differential equations~(PDEs) numerically.
Notable approaches include, e.g., physics-informed neural networks~(PINNs) \cite{Cuomo:2022,RPK:2019}, the deep Ritz method~\cite{EY:2018}, and the finite neuron method~\cite{Xu:2020}.
Each of these methods aims to find optimal parameters for a given neural network by minimizing a given energy functional, essentially solving an optimization problem of the form
\begin{equation}
\label{NN_optimization}
\min_{v \in \mathcal{F}} E(v),
\end{equation}
where $\mathcal{F}$ denotes a class of functions that can be represented by a neural network, and $E \colon \mathcal{F} \rightarrow \mathbb{R}$ is a suitable energy functional.

Despite the superior expressivity of neural networks compared to conventional approaches for numerical PDEs~\cite{HMX:2023,HX:2023}, the optimization problem~\eqref{NN_optimization} is highly non-convex and ill-conditioned. 
In~\cite{HSTX:2022}, it is proven that even training the linear layers in shallow neural networks is ill-conditioned.
This makes it extremely difficult to find a global optimum, even for shallow neural networks, which significantly undermines the application of neural networks in solving PDEs.

Due to the difficulty of training neural networks for solving PDEs mentioned above, various training algorithms tailored for solving PDEs, besides conventional gradient descent-based algorithms~(see, e.g.,~\cite{GBC:2016}), have been designed to tackle optimization problems of the form~\eqref{NN_optimization}.
The hybrid least squares gradient descent method in~\cite{CGPPT:2020} was motivated by the adaptive basis viewpoint.
The active neuron least squares method~\cite{AS:2021,AS:2022}, which manually adjusts position of the neurons according to the training data, was designed to escape from plateaus during training.
The neuron-wise parallel subspace correction method was proposed in~\cite{PXX:2022}, providing a way to precondition the linear layers in shallow neural networks and adjust neurons in a parallel manner.
Random feature methods~\cite{CCEY:2022,EMW:2020,LL:2024,RR:2007} and extreme learning machines~\cite{DL:2021,HZS:2006,LLR:2024} have also gained attention as powerful methods for solving PDEs using neural networks.
Recently, greedy algorithms~\cite{DT:1996,Jones:1992,MZ:1993,Temlyakov:2011} have been successfully applied in neural network optimization; see~\cite{DPW:2021,SHJHX:2023,SX:2022a}.
Most notably, the orthogonal greedy algorithm~(OGA)~\cite{SHJHX:2023,SX:2022a} for solving the optimization problem~\eqref{NN_optimization} has been shown to achieve the optimal approximation rates of shallow neural networks~\cite{SX:2024}.

Despite the guaranteed approximation properties of greedy algorithms, they have a significant drawback; each iteration requires solving a non-convex optimization subproblem~\cite[Section~6]{SHJHX:2023}.
Obtaining an exact optimum for this subproblem is very challenging in general; see~\cite{BM:2002} for some special cases.
Consequently, an approximate solution is typically obtained by searching over a sufficiently fine discretized dictionary of finite size~\cite{SHJHX:2023}.
However, this approach becomes extremely challenging and nearly infeasible in higher dimensions, where the required size of the discretized dictionary becomes exceptionally large.
Therefore, the application of greedy algorithms to higher-dimensional problems is mostly limited due to computational constraints, except for problems where the target functions have very special structures, such as being additively separable~\cite{SHJHX:2023}. 

Given these significant challenges, in this paper, we propose a randomized approach for discretizing dictionaries in greedy algorithms for the optimization of shallow neural networks, with a focus on OGA.
We extend the convergence result of OGA to its weak counterpart~\cite{Temlyakov:2000}, which allow for inexact solutions of the non-convex subproblems. 
One notable contribution is that we generalize the convergence theorem of weak orthogonal greedy algorithm (WOGA) to convex optimization problems. 
We then use these results to analyze our proposed randomized discretization approach.
More precisely, we prove that, with high probability, the proposed approach applied to OGA realizes WOGA, achieving the optimal approximation rate same as the standard OGA.
In addition, we quantify the size of the randomized discrete dictionary that is sufficient to realize WOGA.
The practical performance of the proposed approach is verified by various numerical experiments involving function approximation problems and PDEs, demonstrating orders of magnitude reductions in the size of the discrete dictionary when compared with using a deterministic one. 

The rest of the paper is organized as follows.
In \cref{sec:models}, we introduce our model problems, and discuss relevant approximation results.
In \cref{sec:greedy}, we introduce weak greedy algorithms, and present the convergence result of the WOGA in the setting of convex optimization. 
In \cref{sec:discrete}, we discuss schemes for dictionary discretization and propose a randomized version. 
In \cref{sec:analysis}, we present OGA that uses discrete dictionaries of finite size.
We quantify the required size of a discrete dictionary to guarantee realizations of WOGA, and show that with high probability, our proposed randomized algorithm achieves the same convergence rate of WOGA.
In \cref{sec:experiments}, we demonstrate the effectiveness and broad application of our proposed method through numerous numerical experiments on function approximation problems and elliptic PDEs, and introduce a Fixed-Size Iterative OGA. 
Finally, we give concluding remarks in \cref{sec:conclusion}. 

\section{Neural network optimization}
\label{sec:models}
In this section, we review neural network optimization problems appearing in the finite neuron method~\cite{Xu:2020} and the approximation properties of $\text{ReLU}^k$ shallow neural networks.

\subsection{Finite neuron method}
Let $\Omega \subset \mathbb{R}^d$ be a bounded domain and $V \subset L^2 (\Omega)$ be a Hilbert space consisting of functions on $\Omega$.
Various elliptic boundary value problems can be formulated as optimization problems of the form
\begin{equation}
\label{variational_problem}
\min_{v \in V}  E(v). 
\end{equation}
For example, for the following elliptic PDE: 
\begin{align*}
- \Delta u + u = f & \quad \text{ in } \Omega, \\
\frac{\partial u}{\partial n} =  0 & \quad \text{ on } \partial \Omega,
\end{align*}
the corresponding energy minimization formulation is 
\begin{equation*}
    \min_{v \in V} \left\{ E(v) = \frac{1}{2} a(v,v) - \int_{\Omega} f dx \right\}, 
\end{equation*}
where $a(\cdot, \cdot)$ is a bilinear form given by $$a(u,v) = \int_{\Omega} \left( \nabla u \cdot \nabla v  + uv \right) dx,
\quad u ,v \in H^1 (\Omega).$$
For more general settings on differential operators and boundary conditions, see~\cite{Xu:2020}.

In the finite neuron method, we minimize the energy functional $E$ given in~\eqref{variational_problem} within a class $\mathcal{F}$ of functions that can be represented by a given neural network.
Namely, we solve the minimization problem
\begin{equation}
\label{FNM}
\min_{v \in \mathcal{F}} E(v).
\end{equation}

The convergence analysis of the finite neuron method in terms of the approximation properties of the class $\mathcal{F}$ was provided in~\cite{Xu:2020}.
Additionally, various numerical algorithms for the finite neuron method have been studied in~\cite{JLS:2024,PXX:2022,SHJHX:2023,SX:2022a}.

\subsection{Shallow neural networks}
Here, we discuss the function classes represented by shallow neural networks and their approximation properties.
Recall that a shallow neural network with $n$ neurons, $d$-dimensional input, and $\operatorname{ReLU}^k$ activation~($k \in \mathbb{N} $) is written as
\begin{equation*}
u_n(x) = \sum_{i = 1}^n  a_i \sigma_k (\omega_i \cdot x + b_i),
\quad x \in \mathbb{R}^d,
\end{equation*}
for parameters $\{ a_i \}_{i=1}^n \subset \mathbb{R}$, $\{ \omega_i \}_{i=1}^n \subset S^{d-1}$, and $\{ b_i \}_{i=1}^n \subset \mathbb{R}$, where $\sigma_k$, defined as $\operatorname{ReLU}^k \colon \mathbb{R} \rightarrow \mathbb{R}$, is the activation function 
\begin{equation*}
    \sigma_k (t) = \operatorname{ReLU}(t)^k,\quad \operatorname{ReLU}(t) = {\max}(0, t),
    \quad t \in \mathbb{R},
\end{equation*}
and $S^{d-1}$ represents the unit $d$-sphere in $\mathbb{R}^d$ centered at the origin.
Following~\cite{SHJHX:2023,SX:2024}, we examine the function class $\Sigma_{n, M} (\mathbb{D})$, defined for some $M > 0$, as
\begin{equation}\label{eq:sigma_n_M}
\Sigma_{n,M} (\mathbb{D}) = \left\{ \sum_{i=1}^n a_i d_i : a_i \in \mathbb{R}, \text{ } d_i \in \mathbb{D}, \text{ } \sum_{i = 1}^n |a_i| \leq M \right\},
\end{equation}
where the dictionary $\mathbb{D}$ is symmetric and given by
\begin{equation}
\label{dictionary}
\mathbb{D} = \mathbb{P}_k^d:=
\left\{\pm \sigma_k (\omega \cdot x + b ): \omega \in S^{d-1}, \text{ } b \in\left[c_1, c_2\right]\right\},
\end{equation}
with two constants $c_1$ and $c_2$ chosen to satisfy
\begin{equation}
\label{c1c2}
c_1 \leq \inf \left\{\omega \cdot x : x \in \Omega, \text{ } \omega \in S^{d-1}\right\}
\leq \sup \left\{\omega \cdot x: x \in \Omega, \text{ } \omega \in S^{d-1}\right\} \leq c_2 .
\end{equation}
Throughout this paper, we denote the function class $\mathcal{F}$ in~\eqref{NN_optimization} by $\Sigma_{n, M} (\mathbb{D})$.

Recently, a sharp bound on the approximation rate of $\Sigma_{n, M} (\mathbb{D})$ was established in~\cite{SX:2024}.
We present the approximation results of $\Sigma_{n, M}$ in the space $V = H^m (\Omega)$, where $\Omega \subset \mathbb{R}^d$ be a bounded domain and $H^m (\Omega)$ is the Hilbert space of functions with square-integrable $m$th-order partial derivatives.
For $m = 0$, $H^0(\Omega)$ is the standard $L^2(\Omega)$ space. 

To describe the result from~\cite{SX:2024}, we introduce several mathematical notions.
We first consider the set
\begin{equation}
\label{eq:B1D}
B_1 (\mathbb{D}) = \overline{ \left\{ \sum_{i=1}^n a_i d_i : n \in \mathbb{N}, \text{ } d_i \in \mathbb{D}, \text{ } \sum_{i=1}^n |a_i| \leq 1 \right\} },
\end{equation}
which is the $H^m (\Omega)$-closure of the convex, symmetric hull of $\mathbb{D}$.
The variation norm $\| \cdot \|_{\mathcal{K}_1 (\mathbb{D})}$ with respect to the dictionary $\mathbb{D}$ is defined as
\begin{subequations}
\label{variation_space}
\begin{equation}
\| v \|_{\mathcal{K}_1 (\mathbb{D})} = \inf \left\{ c > 0 : v \in c B_1 (\mathbb{D}) \right\},
\quad v \in H^m (\Omega),
\end{equation}
and the corresponding variation space $\mathcal{K}_1 (\mathbb{D})$ is given by
\begin{equation}
\mathcal{K}_1 (\mathbb{D}) = \left\{ v \in H^m (\Omega) : \| v \|_{\mathcal{K}_1 (\mathbb{D})} < \infty \right\}.
\end{equation}
\end{subequations}
For a detailed characterization of this variation space $\mathcal{K}_1 (\mathbb{D})$, see \cite{SX:2023}.
If the target function $u$ belongs to $\mathcal{K}_1 (\mathbb{D})$, the best approximation $u_n \in \Sigma_{n, M } (\mathbb{D})$ enjoys the following approximation rate.

\begin{proposition}{\cite[Theorems~2 and~3]{SX:2024}}
\label{Prop:rate}
Given $u \in \mathcal{K}_1 (\mathbb{D})$, there exists a positive constant $M$ depending on $k$ and $d$ such that
\begin{equation*}
\inf_{u_n \in \Sigma_{n, M} (\mathbb{D})} \| u - u_n \|_{H^m (\Omega)} \lesssim \| u \|_{\mathcal{K}_1 (\mathbb{D})} n^{-\frac{1}{2} - \frac{2(k-m)+1}{2d}},
\end{equation*}
where $\Sigma_{n, M} (\mathbb{D})$ and $\mathcal{K}_1 (\mathbb{D})$ were given in~\eqref{eq:sigma_n_M} and~\eqref{variation_space}, respectively. 
\end{proposition}

\section{Weak greedy algorithms}
\label{sec:greedy}
Greedy algorithms are well-established methods that have been extensively studied over the past decades for various applications~\cite{DT:1996,Jones:1992,MZ:1993,Temlyakov:2011}.
Recently, the application of greedy algorithms to neural network optimization has been explored in~\cite{DPW:2021,SHJHX:2023,SX:2022a}.
In this section, we briefly introduce the pure greedy algorithms, the weak relaxed greedy algorithm~(WRGA). 
We focus on the weak orthogonal greedy algorithm (WOGA) and provide a general convergence result under the convex optimization setting.
This convergence result will be an important tool in the convergence analysis for the orthogonal greedy algorithms with discrete dictionaries as introduced in~\cref{sec:discrete}. 
Notably, we prove that convergence of WOGA can be extended to convex optimization, and WOGA preserve the same optimal convergence rate as the original greedy algorithm.

We now introduce the greedy algorithms in a setting of minimizing a convex, $L$-smooth energy functional $E \colon V \rightarrow \mathbb{R}$. 
In this section, we assume that the Hilbert space $V$ is a Sobolev space $H^m({\Omega})$ equipped with the inner product $\left< \cdot, \cdot \right>$ and its induced norm $\| \cdot \|$.
 We denote by $\nabla E (v)$ the variational derivative of the energy functional $E$ at $v$. 
Let $\mathbb{D}$ be a symmetric dictionary, that is, if $g \in \mathbb{D}$, then $-g \in \mathbb{D} $. Assume that $\max_{g \in \mathbb{D}} \| g\| = C < \infty$.
Given a current iterate $u_{n-1}$, greedy algorithms involve solving a common subproblem that reads as  
\begin{equation}
    \label{argmax}
    g_n = \operatornamewithlimits{\arg\max}_{g \in \mathbb{D}}  \langle g, \nabla E (u_{n-1}) \rangle. 
\end{equation}
Then, the $n$th step of the pure greedy algorithm (with shrinkage $s > 0$) \cite{Temlyakov:2011, KS:2023} reads as 
\begin{subequations}
\label{GA}
\begin{align}
u_n &= u_{n-1} -s \langle g_n , \nabla E (u_{n-1})\rangle g_n.
\end{align}
\end{subequations}
The $n$th step of the relaxed greedy algorithm (RGA) adopts a relaxation approaches and reads as 
\begin{equation}
    u_n = \left( 1 - \alpha_n \right) u_{n-1} -  \alpha_n M g_n.
\end{equation}
for an appropriately chosen $\alpha_n \in (0,1]$ and $M > 0$. 
We refer the readers to \cite{BCDD:2008,DT:1996,Jones:1992,SHJHX:2023, Zhang:2003} for details on RGA. 
For OGA, we will discuss it in more details later in this section. 

We note that at each iteration of greedy algorithms, we must solve the optimization problem~\eqref{argmax}, which is challenging and computationally intractable for practical implementation.
To make it more ready for practical implementation, one possible strategy is to replace the problem~\eqref{argmax} with a weaker version: we find $g_n \in \mathbb{D}$ such that
\begin{equation}
\label{argmax_weak}
\langle g_n ,\nabla E (u_{n-1}) \rangle  \geq \gamma_n \max_{g \in \mathbb{D}} \langle g, \nabla E (u_{n-1}) \rangle 
\end{equation}
for a given parameter $\gamma_n \in (0, 1]$.
That is, we solve~\eqref{argmax} only approximately so that the value of $\langle g_n, \nabla E (u_{n-1}) \rangle $ closely approximates the maximum of $ \langle  g, \nabla E (u_{n-1}) \rangle $ over $g \in \mathbb{D}$.
Greedy algorithms that adopt~\eqref{argmax_weak} are called \textit{weak} greedy algorithms. 

Now we focus on WOGA under a general setting of minimizing an energy functional $E: V \to \mathbb{R}$ in \cref{alg:WOGA-convex}. 
It is also known as the weak Chebyshev greedy algorithm \cite{DT:2022biorthogonal,Temlyakov:2015greedyconvex}.

\begin{algorithm}[H]
\caption{Weak orthogonal greedy algorithm (WOGA) for convex optimization}
\label{alg:WOGA-convex}
\begin{algorithmic}
    \State Given a dictionary $\mathbb{D}$, $u_0 = 0$, and parameters $\{ \gamma_n \in (0, 1] \}_{n=1}^\infty $,
    \For{$n = 1, 2, \dots$}
        \State Find $g_n \in \mathbb{D}$ such that 
        \begin{equation}
        \label{g_n_WOGA}
         \langle g_n, \nabla E(u_{n-1}) \rangle  \geq \gamma_n \max_{g \in \mathbb{D}} \langle g, \nabla E (u_{n-1}) \rangle.
        \end{equation}
        \State Find $u_n \in H_n :=\operatorname{span} \{ g_1, \dots, g_n \}$ such that
        \begin{equation}
        \label{eq:oga-minimization}
    u_n \in \operatornamewithlimits{\arg\min}_{v \in H_n} E(v).
    \end{equation}
    \EndFor
\end{algorithmic}
\end{algorithm}

Note that~\eqref{eq:oga-minimization} is equivalent to
\begin{equation}\label{eq:oga-variational}
            \langle \nabla E(u_n), w \rangle  = 0
            \quad \forall w \in H_n.
        \end{equation}
In \cref{alg:WOGA-convex}, the weak parameter $\gamma_n$ could vary in each iteration.
If $\gamma_n = \gamma \in (0,1]$ for all $n \geq 1$, then we obtain a special variant of the WOGA.
If $\gamma_n = 1$ for all $n\geq 1$, we obtain the Chebyshev greedy algorithm \cite{Temlyakov:2015greedyconvex}. 

As an important special case, if the energy functional $E(v) = \frac{1}{2} \|u -v \|^2$, we obtain WOGA for function approximation problems. 
The subproblem \eqref{g_n_WOGA} becomes finding $g_n \in \mathbb{D}$ such that 
\begin{equation*} 
    \langle g_n, u - u_{n-1} \rangle  \geq \gamma_n \max_{g \in \mathbb{D}} \langle g, u - u_{n-1}) \rangle.
\end{equation*}
The subproblem~\eqref{eq:oga-minimization} becomes an orthogonal projection onto the subspace $H_n$.

Similarly, if $\gamma_n = \gamma \in (0,1]$, we obtain another version of the WOGA for function approximation. 
If $\gamma = 1$, WOGA reduces to the original OGA \cite{PRK:1993,DT:2019,DT:1996,LS:2024,SX:2022a}.
The convergence analysis of OGA for function approximation was first carried out in~\cite{DT:1996}; see also~\cite{BCDD:2008}. 
For WOGA, the convergence analysis is done in~\cite{Temlyakov:2000}. 
Both analyses lead to an $O(n^{- \frac{1}{2}})$ convergence rate. 
Recently, in~\cite{SX:2022a, LS:2024}, an improved convergence rate was obtained for OGA for function approximation when $B_1 (\mathbb{D})$ has small entropy. The rate is optimal for ReLU$^k$ shallow neural network optimization. 
We follow the proof in~\cite{LS:2024} and generalize the improved convergence results to WOGA for convex optimization in \cref{Thm:WOGA-convex}.
As a result, the convergence of WOGA for function approximation in the Hilbert space $V$ follows as a corollary. 

Before proceeding, we state two lemmas that will play a crucial role in the convergence analysis of WOGA.

\begin{lemma}
\label{Lem:inner_product}
Given $f \in \mathcal{K}_1 (\mathbb{D})$ and $h \in V$, we have
\begin{equation*}
\langle f, h \rangle \leq \| f \|_{\mathcal{K}_1 (\mathbb{D})} \max_{g \in \mathbb{D}} \langle g, h \rangle,
\end{equation*}
where $\mathcal{K}_1 (\mathbb{D})$ was given in~\eqref{variation_space}.
\end{lemma}

The following result follows from the estimate on the metric entropy $\varepsilon_n(B_1(\mathbb{D}))_V  \lesssim  n^{-\frac{1}{2} - \frac{2(k-m) + 1}{2d}}$ for $V = H^{m}(\Omega)$ \cite{SX:2024} and \cite[Lemma 2.1]{LS:2024}. 
\begin{lemma}{\cite[Lemma 2.1]{LS:2024}}
\label{lem:entropy}
Let $\mathbb{D}$ be a $\operatorname{ReLU}^k$ dictionary as defined in \eqref{dictionary} in the space $H^m(\Omega)$, $k \geq m$. 
For any sequence $v_1, \ldots, v_n \in \mathbb{D}$, the following inequality holds:
\begin{equation*}
\left(\prod_{\ell=1}^n \left\| v_\ell - P_{\ell-1} v_\ell \right\|\right)^{\frac{1}{n}} \leq C n^{-\frac{2(k-m)+1}{2d}},
\end{equation*}
where $P_{\ell-1}$ denotes the orthogonal projection onto $\operatorname{span}\{v_1, \dots, v_{\ell-1}\}$, $C > 0$ is a constant that depends on $k$ and the dimension $d$. 
\end{lemma}

Now we present the convergence result for WOGA in \cref{Thm:WOGA-convex}.
\begin{theorem}
    \label{Thm:WOGA-convex}
    Assume that the energy functional $E$ is convex and $L$-smooth  in $H^m(\Omega)$.
    Let $u \in \mathcal{K}_1 (\mathbb{D})$ be a minimizer of $E$.
    In the weak orthogonal greedy algorithm~(\cref{alg:WOGA-convex}), we have
    \begin{equation*}
        E(u_n) - E(u) \leq 2 L C^2 \sqrt[n]{\prod_{\ell = 1}^n \left( \frac{1}{\gamma_\ell^2} \right) }  \| u \|_{\mathcal{K}_1 (\mathbb{D})}^2  
        n^{-1 - \frac{2(k-m)+1}{d}} ,
        \quad n \geq 1,
    \end{equation*}
where the constant $C$ is given in \cref{lem:entropy}.
Furthermore, if the energy functional $E$ is also $\mu$-strongly convex, we have 
\begin{equation*}
    \|u - u_n \| \leq 2 C \sqrt{\frac{L}{\mu}} \sqrt[n]{\prod_{\ell = 1}^n \left( \frac{1}{\gamma_\ell} \right) } \| u \|_{\mathcal{K}_1 (\mathbb{D})} n^{-\frac{1}{2} - \frac{2(k-m)+1}{2d}},
        \quad n \geq 1
\end{equation*}
If $\gamma_\ell = \gamma \in (0,1]$ for all $\ell \leq n$, we have 
\begin{equation*}
    \|u - u_n \| \leq \frac{2C}{\gamma} \sqrt{\frac{L}{\mu}} \| u \|_{\mathcal{K}_1 (\mathbb{D})} n^{-\frac{1}{2} - \frac{2(k-m)+1}{2d}},
        \quad n \geq 1 
\end{equation*}
\end{theorem}
\begin{proof}
For a given $n \geq 1$, let $\bar{g}_n = g_n - P_{n-1} g_n$, where $P_{n-1}$ denotes the orthogonal projection onto $H_{n-1}$.
It is clear that $\bar{g}_n \in H_n$.

Now, by optimality \eqref{eq:oga-minimization} and $L$-smoothness, we obtain
\begin{equation}\label{Thm2:OGA-convex}
    E(u_n)
    \stackrel{\eqref{eq:oga-minimization}}{\leq} E(u_{n-1} - \bar{\alpha} \bar{g}_n ) \\
    \leq E(u_{n-1}) - \bar{\alpha} \langle \nabla E (u_{n-1}), \bar{g}_n  \rangle + \frac{L}{2} \bar{\alpha}^2 \| \bar{g}_n \|^2
    = E(u_{n-1}) - \frac{\langle \nabla E (u_{n-1}), \bar{g}_n  \rangle^2}{2L \| \bar{g}_n \|^2},
\end{equation}
where the equality is obtained by setting 
$\bar{\alpha} = \frac{\langle \nabla E(u_{n-1}), \bar{g}_n \rangle}{L \| \bar{g}_n \|^2}$ . \\

Moreover, for any $v \in \mathcal{K}_1 (\mathbb{D})$ we have
\begin{multline}
\label{Thm1:OGA-convex}
\langle \nabla E(u_{n-1}), \bar{g}_n \rangle
\stackrel{\eqref{eq:oga-variational}}{=} \langle \nabla E(u_{n-1}), g_n \rangle
\geq \gamma_n \max_{g \in \mathbb{D}} \langle \nabla E(u_{n-1}), g \rangle \\
\stackrel{\text{(i)}}\geq  \frac{\gamma_n \langle \nabla E(u_{n-1}), -v \rangle}{\| v \|_{\mathcal{K}_1 (\mathbb{D})}}
\stackrel{\eqref{eq:oga-variational}}{=} \frac{  \gamma_n \langle \nabla E (u_{n-1}), u_{n-1} - v \rangle}{\| v \|_{\mathcal{K}_1 (\mathbb{D})} }
\stackrel{\text{(ii)}}{\geq} \frac{\gamma_n ( E(u_{n-1}) - E(v) )}{\| v \|_{\mathcal{K}_1 (\mathbb{D})}},
\end{multline}
where (i) and (ii) are due to \cref{Lem:inner_product} and the convexity of $F$, respectively.

Combining~\eqref{Thm2:OGA-convex} and~\eqref{Thm1:OGA-convex}, we obtain 
\begin{equation*}
    E(u_n) - E(v) \leq E(u_{n-1}) - E(v) - \frac{ \gamma_n^2 ( E(u_{n-1}) - E(v) )^2}{2 L \| v \|_{\mathcal{K}_1 (\mathbb{D})}^2 \| \bar{g}_n \|^2},
\end{equation*}
from which we define
\begin{equation*}
        a_n = \frac{ E(u_n) - E(v) }{2 L \|v\|_{\mathcal{K}_1 (\mathbb{D})}^2}, \quad
        b_n = \gamma_n^2  \|\bar{g}_n \|^{-2},
\end{equation*}
to obtain a recurrence relation $a_n \leq a_{n-1} (1 - b_n a_{n-1})$ . 

We then apply the same argument as in \cite[Theorem 4.2]{LS:2024} and invoke \cref{lem:entropy} to obtain  
\begin{equation}
    a_n \leq \frac{1}{n} \sqrt[n]{\prod_{\ell = 1}^n\frac{1}{\gamma_\ell^2}} \sqrt[n]{ \prod_{\ell = 1}^n \| \bar{g}_\ell\|^2 }
    \leq C^2 \sqrt[n]{\prod_{\ell = 1}^n \left( \frac{1}{\gamma_\ell^2} \right) } n^{-1 - \frac{2(k-m)+1}{d}}.
\end{equation}
This gives us the first estimate.

Now, if the energy functional is $\mu$-strongly convex, we have 
\begin{equation*}
E(u_n) - E(u) \geq \langle \nabla E(u), u_n - u \rangle + \frac{\mu}{2} \| u_n - u \|^2 = \frac{\mu}{2} \| u_n - u \|^2,
\end{equation*}
for any $n \geq 1$.
This gives us the estimate when the energy functional is $\mu$-strongly convex and $L$-smooth. 

Finally, we can easily obtain the last estimate by setting $\gamma_\ell = \gamma$ for all $\ell \leq n$. 
\end{proof}

\begin{remark}
    As a special case, we consider the WOGA for function approximation problems. 
    The corresponding energy functional $E(v) = \frac{1}{2}\|v - u\|^2$ is strongly convex with $\mu = 1$ and smooth with $L = 1$.
    Based on \cref{Thm:WOGA-convex}, we can easily obtain the optimal convergence rate for WOGA with weak parameter $\gamma_n = \gamma \in (0,1]$. 
    In fact, we can slightly refine the constant coefficient in the convergence rate for function approximation problems
    \begin{equation*}
    \left\| u - u_n \right\| \leq \frac{C}{\gamma} \| u \|_{\mathcal{K}_1(\mathbb{D})} n^{ - \frac{1}{2} - \frac{2(k-m) + 1}{2d}}.
    \end{equation*}
\end{remark}

\begin{remark}
    From \cref{Thm:WOGA-convex}, we see that if $\gamma_\ell = \gamma \in (0,1]$, then we preserve the same optimal convergence rate as the original orthogonal greedy algorithm for function approximation \cite{SX:2022a,LS:2024}. 
    Additional, more insight can be obtained from \cref{Thm:WOGA-convex} that even if $\{\gamma_\ell\}_{\ell = 1}^\infty$ is a decreasing sequence, we can still have convergence of the weak orthogonal greedy algorithm but with a lower convergence rate if $\{\gamma_\ell\}_{\ell = 1}^\infty$ is not decreasing too rapidly. 
    
    For instance, if we have $\gamma_\ell = \ell^{-\alpha}$, $\alpha > 0$, we obtain the following convergence rate from \cref{Thm:WOGA-convex}
    \begin{equation*}
        \|u - u_n \|_{H^m(\Omega)} \leq 2 C \sqrt{\frac{L}{\mu}} (n!)^{\frac{\alpha}{n}} \| u \|_{\mathcal{K}_1 (\mathbb{D})} n^{- \frac{1}{2} - \frac{2(k-m) + 1 }{2d}}.
    \end{equation*}
    By the Stirling's formula, we obtain 
\begin{equation}
     \|u -u_n \|_{H^m(\Omega)} \lesssim \| u \|_{\mathcal{K}_1 (\mathbb{D})}   n^{- \frac{1}{2} - \frac{2(k-m) +1}{2d} +\alpha }, \quad n \geq 1,\quad u \in \mathcal{K}_1 (\mathbb{D}). 
\end{equation}
It is then easy to see that if $\alpha < \frac{1}{2}+ \frac{2(k-m)+1}{2d}$, we still have a convergent algorithm.
\end{remark}

\section{Discretized dictionaries}
\label{sec:discrete}
In this section, we introduce the scheme of a deterministic discrete dictionary to approximately solve the $\operatorname{\arg\max}$ subproblem \eqref{argmax} in the greedy algorithms, and then propose its randomized variant.
 
\subsection{Deterministic discretization}
In~\cite[Section~6.2]{SHJHX:2023}, discretizing the dictionary $\mathbb{D}$ using a regular grid on polar and spherical coordinates for parametrization of the dictionary in two- and three-dimensions, respectively, was proposed.
Here, we present a generalization of the scheme for arbitrary dimensions.

From the definition~\eqref{dictionary} of $\mathbb{D}$, we parametrize the unit $d$-sphere $S^{d-1}$ in order to obtain a parametrization of $\mathbb{D}$.
In the cases $d = 2$ and $d = 3$, we may use the polar and spherical coordinates to parametrize $S^{d-1}$, respectively, as discussed in~\cite{SHJHX:2023}.
Accordingly, for higher dimensions, it is natural to make use of the hyperspherical coordinates~\cite{Blumenson:1960} to parametrize $S^{d-1}$.
Namely, for $\omega = (\omega_1, \dots, \omega_d) \in S^{d-1}$, we have $\omega = \mathcal{S} (\phi)$ for some $\phi = (\phi_1, \dots, \phi_{d-1}) \in [0, \pi]^{d-2} \times [0, 2\pi) \subset \mathbb{R}^{d-1}$, where the hyperspherical coordinate map $\mathcal{S} \colon [0, \pi]^{d-2} \times [0, 2\pi) \rightarrow S^{d-1}$ is defined as
\begin{equation}\label{eq:hyperspherical}
\begin{aligned}
(\mathcal{S}(\phi))_i &= \left( \prod_{j=1}^{i-1} \sin \phi_j \right) \cos \phi_i,
\quad 1 \leq i \leq d-1, \\
(\mathcal{S}(\phi))_d &= \prod_{j=1}^{d-1} \sin \phi_j,
\end{aligned}
\quad\quad \phi \in [0, \pi]^{d-2} \times [0, 2\pi).
\end{equation}
A useful property of the hyperspherical coordinate map $\mathcal{S}$ is that it is nonexpansive, as stated in \cref{lemma:nonexpansive}.

\begin{lemma}
\label{lemma:nonexpansive}
    The hyperspherical coordinate map $\mathcal{S} \colon [0, \pi]^{d-2} \times [0, 2\pi) \rightarrow S^{d-1}$ defined in~\eqref{eq:hyperspherical} is nonexpansive.
    That is, we have
    \begin{equation*}
    | \mathcal{S} (\phi) - \mathcal{S} (\hat{\phi}) |
    \leq | \phi - \hat{\phi} |,
    \quad \phi, \hat{\phi} \in [0, \pi]^{d-2} \times [0, 2\pi).
    \end{equation*}
\end{lemma}
\begin{proof}
We readily obtain the desired result by invoking the following identity:
\begin{equation*}
| \mathcal{S} (\phi) - \mathcal{S} (\hat{\phi}) |^2
= 4 \sum_{i=1}^{d-1} \left( \prod_{j=1}^{i-1} \sin \phi_j \sin \hat{\phi}_j \right) \sin^2 \frac{\phi_i - \hat{\phi}_i}{2} ,
\quad \phi, \hat{\phi} \in [0, \pi]^{d-2} \times [0, 2\pi),
\end{equation*}
which can be proven without major difficulty. 
\end{proof}

With the help of the hyperspherical coordinate map $\mathcal{S}$, the dictionary $\mathbb{D}$ can be parametrized as follows:
\begin{equation}
\label{parametrization}
\mathcal{P} \colon R \rightarrow \mathbb{D},
\quad
\mathcal{P}(\phi, b) = \sigma_k (\mathcal{S}(\phi) \cdot x + b). 
\end{equation}
The parameter space $R$ is defined as
\begin{equation}
\label{parameter_space}
R = \left( [0, \pi]^{d-2} \times [0, 2\pi) \right) \times [c_1, c_2].
\end{equation}
We denote $|R|$ as the volume of the parameter space $R$, given by $|R| = 2 \pi^{d-1} (c_2 - c_1)$.

Given a bounded domain $\Omega$, we can deduce the map $\mathcal{P}$ is Lipschitz continuous, i.e.,
\begin{equation}
\label{Lipschitz}
\| \mathcal{P}(\phi, b) - \mathcal{P}(\hat{\phi}, \hat{b}) \| \leq \operatorname{Lip}(\mathcal{P}) | (\phi, b) - (\hat{\phi}, \hat{b}) |,
\quad (\phi, b), (\hat{\phi}, \hat{b}) \in R,
\end{equation}
where the Lipschitz modulus $\operatorname{Lip} (\mathcal{P})$ depends on $d$, $k$, $\Omega$, $c_1$, and $c_2$. See \cref{rem:Lipschitz}.

\begin{remark}
\label{rem:Lipschitz}
To examine the dependence of the Lipschitz modulus $\operatorname{Lip} (\mathcal{P})$ on the problem setting, we conduct an analysis of $\operatorname{Lip} (\mathcal{P})$, assuming $\Omega = (0,1)^d \subset \mathbb{R}^d$.
We observe that setting $c_1 = -\sqrt{d}$ and $c_2 = \sqrt{d}$ in the definition~\eqref{dictionary} of $\mathbb{D}$ satisfies the condition~\eqref{c1c2}.
Namely, we have
\begin{equation*}
| \omega \cdot x | \leq | \omega | \, | x |
\leq \sqrt{1^2 + \dots + 1^2} = \sqrt{d}.
\end{equation*}
In this setting, we see that $\omega \cdot x + b \in [-2 \sqrt{d}, 2 \sqrt{d}]$ for any $(\omega, b) \in S^{d-1} \times [c_1, c_2]$ and $x \in \Omega$.
On the interval $[-2 \sqrt{d}, 2 \sqrt{d}]$, the activation function $\sigma_k$ is Lipschitz continuous with modulus
\begin{equation}
\label{Lip_sigma}
\max_{t \in [-2 \sqrt{d}, 2 \sqrt{d}]} | \sigma_k' (t) | = k(2 \sqrt{d})^{k-1}.
\end{equation}
Hence, for any $(\phi, b), (\hat{\phi}, \hat{b}) \in R$, we have
\begin{equation*}
\begin{split}
\| \mathcal{P}(\phi, b) - \mathcal{P}(\hat{\phi}, \hat{b}) \|^2
&= \int_{\Omega} \left| \sigma_k (\mathcal{S} (\phi \cdot x + b)) - \sigma_k (\mathcal{S} (\hat{\phi} \cdot x + \hat{b})) \right|^2 \,dx \\
&\stackrel{\eqref{Lip_sigma}}{\leq} k^2 (4d)^{k-1} \int_{\Omega} \left| ( \mathcal{S} (\phi) - \mathcal{S} (\hat{\phi}) ) \cdot x + (b - \hat{b}) \right|^2 \,dx \\
&\leq k^2 (4d)^{k-1} \left( ( \mathcal{S} (\phi) - \mathcal{S} (\hat{\phi}) )^2 + (b - \hat{b})^2 \right) \int_{\Omega} (|x|^2 + 1) \,dx \\
&\leq k^2 (4d)^{k-1} \left( \frac{d}{3} + 1 \right) | (\phi, b) - (\hat{\phi}, \hat{b}) |^2,
\end{split}
\end{equation*}
where the last inequality is due to \cref{lemma:nonexpansive}.
Therefore, we conclude that
\begin{equation*}
\operatorname{Lip}(\mathcal{P}) \leq k (2 \sqrt{d})^{k-1} \sqrt{\frac{d}{3} + 1}.
\end{equation*}
\end{remark}

To construct a deterministic discrete dictionary $\mathbb{D}_N$ of size $N$, we use a regular hyperrectangular grid on $R$, and then apply the map $\mathcal{P}$ given in~\eqref{parametrization}.
We have
\begin{equation}\label{eq:deterministic-dict}
\mathbb{D}_N = \left\{ \pm \sigma_k (\mathcal{S}(\phi) \cdot x + b) : (\phi, b) \in G_N \right\},
\end{equation}
where $G_N$ is defined as the set of centers of the cells in a regular hyperrectangular grid that divides $R$ into $N$ cells. 
Using elementary inequalities, it is straightforward to verify that the length $\ell$ of the main diagonal of each cell in the grid $G_N$ satisfies
\begin{equation*}
\ell \geq \sqrt{d} \left( \frac{|R|}{N} \right)^{\frac{1}{d}}.
\end{equation*}
We define a constant $\delta \geq 1$ as
\begin{equation}
\label{delta}
\delta = \frac{\ell}{\sqrt{d}} \left( \frac{N}{|R|} \right)^{\frac{1}{d}}.
\end{equation}

\subsection{Random discretization}
Now, we consider a random discretization of the dictionary $\mathbb{D}$.
One straightforward way of randomly discretizing the dictionary is to utilize the hyperspherical coordinate map \eqref{eq:hyperspherical}. we draw $N$ samples from the uniform distribution on the parameter space $R$, and then apply the parametrization map $\mathcal{P}$ given in~\eqref{parametrization}.
Therefore, a randomized discrete dictionary $\mathbb{D}_N$ is defined as
\begin{equation}
\label{eq:randomized-dict}
\mathbb{D}_N = \left\{\pm \sigma_k (\mathcal{S} (\phi_i) \cdot x + b_i) : (\phi_i, b_i) \sim \mathrm{Uniform}(R), \text{ } 1 \leq i \leq N \right\}. 
\end{equation}

\begin{remark}
\label{rem:uniform_sphere}
An alternative approach to defining a randomized discrete dictionary is to directly sample points from the uniform distribution on $S^{d-1} \times [c_1, c_2]$ directly, without relying on its parametrization.
Generating uniform random variables defined on a hypersphere can be achieved by appropriately manipulating normal random variables, see~\cite{Marsaglia:1972,Muller:1959}. 
However, generating deterministic points that are evenly spaced on  $S^{d-1}$ is less obvious in higher dimensions.
For the sake of simplicity and consistency in our analysis, we adhere to the method that utilizes the parametrization $\mathcal{P}$ for both deterministic and randomized discretizations.
\begin{remark}\label{Rem:optimalCovering}
It is natural to seek for a discrete dictionary of minimal size for a desired accuracy.
In our case, this is equivalent to finding an arrangement of points that forms a covering of the parameter space $R$ with congruent balls.
However, constructing an ideal arrangement of points is a highly nontrivial task in different dimensions, especially in higher ones \cite{Fejes:2022,FK:1993}. 
Covering density, roughly defined as the ratio of sum of the Lebesgue measures of the bodies and that of the domain being covered, measures the efficiency of a covering. 
On one hand, although the upper bound (\cite{rogers1957note, rogers1959lattice}) and the lower bound (\cite{coxeter1959covering}) for the covering density with congruent balls are available, the proofs are non-constructive, meaning that we do not know how to construct an arrangement of points that can achieve the lower bound in high dimensions.  
On the other hand, the optimal covering of the Euclidean space $\mathbb{R}^d$ with balls is only known for $d \leq 5$ \cite{kershner1939number, Bambah:1954, DR:1963, RB:1975, ryshkov1978}. 
Among the lattice arrangement covering, for which our deterministic discretization belongs to, the optimal covering is only known in 2D and is related to the regular hexagonal pattern \cite{kershner1939number}. 
The usage of the regular hyperrectangular grid introduced in \eqref{eq:deterministic-dict} is due to highly non-trivial explicit construction of efficient coverings in different dimensions.
\end{remark}

\end{remark}

\section{Analysis of dictionary sizes}
\label{sec:analysis}
In this section, we present a rigorous analysis demonstrating that utilizing discrete dictionaries for OGA, as introduced in the previous section, can yield convergence rates same as the original OGA, provided that the size of the discrete dictionary is sufficiently large. 
Throughout this section, we will focus on OGA for convex optimization, and we assume that the energy functional $E$ is $\mu$-strongly convex and $L$-smooth in $V = L^2(\Omega)$. 
The convergence analysis under other Hilbert spaces is similar. 

First, in \cref{alg:OGA-deterministic-dictionary}, we present OGA that uses a deterministic discrete dictionary $\mathbb{D}_N$ defined in~\eqref{eq:deterministic-dict}.

\begin{algorithm}[H]
\caption{OGA with a deterministic discrete dictionary}
\label{alg:OGA-deterministic-dictionary}
\begin{algorithmic}
    \State Given a deterministic discrete dictionary $\mathbb{D}_N$ of size $N$ as defined in~\eqref{eq:deterministic-dict}, and $u_0 = 0$,

    \For{$n = 1, 2, \dots$}
        \State $\displaystyle
        g_n = \operatornamewithlimits{\arg\max}_{g \in \mathbb{D}_N}  \langle g, \nabla E(u_{n-1}) \rangle $
        \State Find $u_n \in H_n :=\operatorname{span} \{ g_1, \dots, g_n \}$ such that
        \begin{equation*}
    u_n \in \operatornamewithlimits{\arg\min}_{v \in H_n} E(v).
    \end{equation*}
    \EndFor
\end{algorithmic}
\end{algorithm}

In \cref{thm:deterministic}, we state a lower bound for the size $N$ of the deterministic discrete dictionary $\mathbb{D}_N$ that ensures the optimal convergence rate of the error $\| u - u_n \|$. 

\begin{theorem}
\label{thm:deterministic}
Let $\gamma \in (0, 1)$. Assume that the energy functional $E$ is $\mu$-strongly convex and $L$-smooth.
In OGA with the deterministic discrete dictionary $\mathbb{D}_N$~(see \cref{alg:OGA-deterministic-dictionary}), for any $n \geq 1$, if
\begin{equation}
\label{thm1:deterministic}
N \geq |R| \left( \frac{\gamma}{1 - \gamma} \frac{\delta \sqrt{d} \operatorname{Lip} (\mathcal{P})}{4 C} \sqrt{\frac{L}{\mu}} \right)^d n^{ (\frac{1}{2} + \frac{2k+1}{2d})d} ,
\end{equation}
then we have
\begin{equation*}
\| u - u_n \| \leq \frac{2C}{\gamma} \sqrt{\frac{L}{\mu}} 
 \|  u \|_{\mathcal{K}_1 (\mathbb{D})} n^{-\frac{1}{2} - \frac{2k+1}{2d}},
\end{equation*}
where $\mathcal{K}_1 (\mathbb{D})$, $C$, $R$, $\operatorname{Lip} (\mathcal{P})$, and $\delta$ were given in~\eqref{variation_space},~\cref{Thm:WOGA-convex},~\eqref{parameter_space},~\eqref{Lipschitz}, and~\eqref{delta}, respectively.
\end{theorem}

\begin{proof}
    As the sequence $\{E(u_m) -E(u) \}$ is decreasing in \cref{alg:OGA-deterministic-dictionary}, we may assume the following: 
    \begin{equation}\label{eq:E_lower}
    E(u_m) - E(u) \geq \frac{2C^2 L}{\gamma^2} \|u\|^2_{\mathcal{K}_1 (\mathbb{D})} n^{-1 - \frac{2k+1}{d}}, \quad 0\leq m\leq n-1.
    \end{equation}
Now we suppose the dictionary size $N$ is given by \eqref{thm1:deterministic}.
Let $(\phi_m^*, b_m^*)$ be an element of $R$ such that $g = \mathcal{P}(\phi_m^*, b_m^*)$ maximizes $| \langle g, \nabla E(u_{m-1}) \rangle |$ in $\mathbb{D}$.
Recall that the length of the main diagonal of each cell in the grid $G_N$ is denoted by $\ell$.
Then we can find a grid point $(\phi_m, b_m) \in G_N$ such that
\begin{equation}
\label{thm3:deterministic}
| (\phi_m, b_m) - (\phi_m^*, b_m^*) |
\leq \frac{\ell}{2} \leq \frac{1-\gamma}{\gamma} \frac{2 C}{\operatorname{Lip}(\mathcal{P})} \sqrt{\frac{\mu}{L}} n^{-\frac{1}{2} - \frac{2k+1}{2d}}.
\end{equation}

It follows that
\begin{equation}
\label{thm5:deterministic}
\begin{split}
| \langle \mathcal{P} (\phi_m^*, b_m^*), & \nabla E(u_{m-1})\rangle | -
 | \langle \mathcal{P} (\phi_m, b_m), \nabla E(u_{m-1}) \rangle | \\ 
&\leq \left| \langle \mathcal{P} (\phi_m, b_m) - \mathcal{P} (\phi_m^*, b_m^*), \nabla E(u_{m-1}) \rangle \right| \\
&\stackrel{\eqref{Lipschitz}}{\leq} \operatorname{Lip} (\mathcal{P}) | (\phi_m, b_m) - (\phi_m^*, b_m^*) | \, \| \nabla E(u_{m-1}) \| \\
&\stackrel{(*)}{\leq} \operatorname{Lip} (\mathcal{P}) | (\phi_m, b_m) - (\phi_m^*, b_m^*)| \frac{ L \| u \|_{\mathcal{K}_1 (\mathbb{D})} | \langle \mathcal{P} (\phi_m^*, b_m^*), \nabla E(u_{m-1}) \rangle |}{\| \nabla E(u_{m-1}) \|} \\
& \stackrel{(**)}{\leq}    \frac{ L \| u \|_{\mathcal{K}_1 (\mathbb{D})} \operatorname{Lip} (\mathcal{P}) | (\phi_m, b_m) - (\phi_m^*, b_m^*)|}{\sqrt{2 \mu (E(u_{m-1}) - E(u))}} | \langle \mathcal{P} (\phi_m^*, b_m^*), \nabla E(u_{m-1}) \rangle | \\ 
&\stackrel{\eqref{eq:E_lower}}{\leq} \left( 1 - \gamma \right) | \langle \mathcal{P} (\phi_m^*, b_m^*), \nabla E(u_{m-1}) \rangle |,
\end{split}
\end{equation}
where $(*)$ is due to the following:
\begin{equation*}
\frac{1}{L} \|\nabla E(u_{m-1}) \|^2 \leq |\langle \nabla E(u_{m-1}) -\nabla E(u) , u_{m-1} -u \rangle | =  |\langle \nabla E(u_{m-1}) , u \rangle | \leq  \|u\|_{\mathcal{K}_1(\mathbb{D})} \max_{g \in \mathbb{D}} \langle g, \nabla E(u_{m-1}) \rangle ,  
\end{equation*}
and $(**)$ is due to
\begin{equation*}
    E(u_{m-1}) - E(u) \leq \frac{1}{2\mu} \|\nabla E(u_{m-1}) \|^2. 
\end{equation*}

Rearranging~\eqref{thm5:deterministic} , we obtain
\begin{equation}
\label{thm6:deterministic}
| \langle \mathcal{P} (\phi_m, b_m), \nabla E(u_{m-1}) \rangle |
\geq \gamma | \langle \mathcal{P} (\phi_m^*, b_m^*),  \nabla E(u_{m-1}) \rangle |,
\end{equation}
which implies that the initial $n$ iterations of \cref{alg:OGA-deterministic-dictionary} realize WOGA~(see \cref{alg:WOGA-convex}).
Therefore, invoking \cref{Thm:WOGA-convex} yields the desired result.
\end{proof}

Next, the approach that uses a randomized discrete dictionary~\eqref{eq:randomized-dict} at each iteration of OGA is given in \cref{alg:OGA-randomized-dictionary}. 
This approach is also used in~\cite{CDDN:2020} for greedy algorithms in reduced basis methods. 
In practice, one can also further optimize the $\arg\max$ subproblem using a few steps of a damped Newton's method on a selected subset of elements from the randomized dictionary.
As demonstrated in \ref{sec:local-optimization} (where the local optimization algorithm is presented), this additional local optimization could further improve the accuracy. 
We will also see from experiments in \cref{sec:experiments} that using randomized dictionaries is also very effective.
\begin{algorithm}
\caption{OGA with randomized discrete dictionaries}
\label{alg:OGA-randomized-dictionary}
\begin{algorithmic}
    \State Given $u_0 = 0$,

    \For{$n = 1, 2, \dots$}
        \State Sample $(\phi_{i,n}, b_{i,n}) \sim \mathrm{Uniform}(R)$ for $i = 1, ..., N$ to generate a randomized discrete dictionary
        \begin{equation*}
            \mathbb{D}_{N, n} = \left\{ \pm \sigma_k ( \mathcal{S} (\phi_{i, n}) \cdot x + b_{i, n} ) : 1 \leq i \leq N \right\}.
        \end{equation*}
        \State $\displaystyle
        g_n = \operatornamewithlimits{\arg\max}_{g \in \mathbb{D}_{N,n}} \langle g, \nabla E(u_{n-1}) \rangle $
        \State Find $u_n \in H_n :=\operatorname{span} \{ g_1, \dots, g_n \}$ such that
        \begin{equation*}
    u_n \in \operatornamewithlimits{\arg\min}_{v \in H_n} E(v).
    \end{equation*}
    \EndFor
\end{algorithmic}
\end{algorithm}

\Cref{thm:randomized} provides a lower bound for the size of the randomized discrete dictionaries $\{ \mathbb{D}_{N,n} \}$ that ensures the optimal convergence rate of the residual error $\| u - u_n \|$ presented in~\cref{Thm:WOGA-convex}. 

\begin{lemma}\label{lem:probability}
    Given a ball of radius $\epsilon$ contained in the parameter space $R$ as defined in \eqref{parameter_space} , if we sample $N$ points uniformly on $R$, the probability of having at least one point in the ball is given by $1 - (1-p)^N$, 
    where $$p = \frac{V_d \epsilon^d}{|R|},$$
    and $V_d = \frac{\pi^{\frac{d}{2}}}{\Gamma(\frac{d}{2} + 1)}$ is the volume of an $\ell^2$-unit ball in $\mathbb{R}^d$.
\end{lemma}
\begin{theorem}
\label{thm:randomized}
Let $\gamma, \eta \in (0, 1)$.
In OGA with the randomized discrete dictionaries $\{ \mathbb{D}_{N, n} \}_n$~(see \cref{alg:OGA-randomized-dictionary}), for any $n \geq 1$, if
\begin{equation}
\label{thm1:randomized}
N \geq  \left( |R| \cdot \frac{1}{V_d} \left( \frac{\gamma}{1-\gamma} \frac{\operatorname{Lip} (\mathcal{P})}{2 C} \sqrt{\frac{L}{\mu}}\right)^d n^{(\frac{1}{2} + \frac{2k+1}{2d})d} -\frac{1}{2} \right)  \log \frac{n}{\eta} ,
\end{equation}
then, with probability at least $1 - \eta$, we have
\begin{equation*}
\| u - u_n \| \leq \frac{2C}{\gamma} \sqrt{\frac{L}{\mu}} \| u \|_{\mathcal{K}_1 (\mathbb{D})} n^{-\frac{1}{2} - \frac{2k+1}{2d}},
\end{equation*}
where $\mathcal{K}_1 (\mathbb{D})$, $C$, $R$, and $\operatorname{Lip} (\mathcal{P})$ were given in~\eqref{variation_space},~\cref{Thm:WOGA-convex},~\eqref{parameter_space}, and~\eqref{Lipschitz}, respectively.
\end{theorem}
\begin{proof}
Similar to in \cref{thm:deterministic}, we may assume that~\eqref{eq:E_lower} holds.
For each $m$ with $1 \leq m \leq n-1$, let $(\phi_m^*, b_m^*)$ be an element of $R$ such that $g = \mathcal{P} (\phi_m^*, b_m^*)$ maximizes $| \langle g, \nabla E(u_{m-1}) \rangle |$ in $\mathbb{D}$.
By the same argument as in~\eqref{thm5:deterministic}, we can prove that for each $(\phi_m, b_m) \in R$,~\eqref{thm3:deterministic} implies~\eqref{thm6:deterministic}.
Accordingly, we conclude the following:
\begin{quote}
$(*)$ \quad If there exists an element $(\phi_m, b_m) \in \mathbb{D}_{N, m}$  such that $$|(\phi_m, b_m) - (\phi_m^*, b_m^*) | \leq \frac{1 - \gamma}{\gamma} \frac{2 C}{\operatorname{Lip} (\mathcal{P})} \sqrt{\frac{\mu}{L}} n^{-\frac{1}{2} - \frac{2k+1}{2d}} =: \epsilon$$ for each $0 \leq m \leq n-1$, then the initial $n$ iterations of \cref{alg:OGA-randomized-dictionary} realize WOGA with a weak parameter $\gamma$~(see \cref{alg:WOGA-convex}).
\end{quote}
Therefore, if the condition~$(*)$ holds, invoking \cref{Thm:WOGA-convex} yields the desired result.

Next, we show that~\eqref{thm1:randomized} implies that~$(*)$ holds with probability $1 - \eta$.
By \cref{lem:probability}, the probability that the premise of~$(*)$ holds is given by
\begin{equation*}
\left( 1 - (1- p)^N \right)^m,
\end{equation*}
where $p$ is the ratio between the volume of a $d$-dimensional ball with radius $\epsilon$ and that of the entire parameter space $R$, i.e.,
\begin{equation*}
p = \frac{V_d \epsilon^d}{|R|}.
\end{equation*}
Now, we suppose that $N$ satisfies~\eqref{thm1:randomized}.
Then we have
\begin{equation}
\label{thm2:randomized}
N 
\geq \left( \frac{1}{p} - \frac{1}{2} \right) \log \frac{n}{\eta}
\geq \frac{\log \frac{n}{\eta}}{\log \frac{1}{1-p}},
\end{equation}
where the second inequality is due to an elementary inequality~(see, e.g.,~\cite[equation~(3.5)]{Park:2022})
\begin{equation*}
\log \left( 1 + \frac{1}{t} \right) \geq \frac{2}{2t+1},
\quad t > 0.
\end{equation*}
By direct calculation, we verify that~\eqref{thm2:randomized} implies
\begin{equation*}
\left( 1 - (1-p)^N \right)^m \geq  1 - m(1-p)^N \geq 1 - n(1-p)^N = 1 - e^{\ln(n) + N \ln(1-p)}\geq 1 - \eta,
\end{equation*}
which completes the proof.
\end{proof}

\begin{remark}\label{Rem:other-activation}
Although \cref{alg:OGA-randomized-dictionary} is tailored for training $\operatorname{ReLU}^k$ shallow neural networks, the strategy of employing randomized dictionaries also works for other activation functions with non-compact parameter spaces, e.g., tanh, sigmoid, etc. In these situations, we can still randomly sample the parameters to form a randomized discrete dictionary, e.g., using a Gaussian distribution or a uniform distribution on a truncated parameter space, see \cite[Example 6]{SHJHX:2023}. 
\end{remark}

\section{Numerical results}
\label{sec:experiments}
In this section, we present numerical results for $L^2$-minimization problems and the Neumann problem in various dimensions to demonstrate the effectiveness of our method of using randomized dictionaries. 
In addition, we solve a Poisson--Boltzmann equation using OGA with randomized dictionaries, illustrating its good convergence properties for nonlinear PDEs.
Furthermore, we also address an indefinite elliptic PDE to show that OGA achieves good convergence for an indefinite problem even in the absence of a theoretical guarantee at this stage.
We also explore the application of OGA with a $\tanh$ activation function to this problem, as discussed in \cref{Rem:other-activation}, and achieve promising accuracy. 
Finally, we modify OGA to propose an iterative version, demonstrating its potential with an $L^2$-minimization example.

Numerical integration is needed for evaluating inner product in the $\operatorname{\arg\max}$ subproblem and for assembling linear systems in the orthogonal projection subproblem. 
In 1D, 2D and 3D, we use the Gauss-Legendre quadrature rule.
For higher dimensions, we use quasi-Monte Carlo integration~\cite{Caflisch:1998,MC:1995}. 
We do not investigate the effect of the number of sampling points on the stability of the orthogonal projection. 
See \cite{CDL:2013, MNVT:2014, CCMNT:2015} for relevant studies on least squares approximations using polynomials. 
The number of sampling points, chosen to ensure sufficient accuracy, is specified in each experiment.

All experiments are implemented using Python and PyTorch, and run on a NVIDIA A100 GPU.

\subsection{$L^2$-minimization problems}
Through solving the following $L^2$-minimization problems in different dimensions using two different discretization schemes for the dictionary, we demonstrate the effectiveness of our proposed method of using randomized OGA.

In 2D, we fit a Gabor function (see Figure \ref{fig:2dgabor}) given by
\begin{equation}\label{eq:2d-gabor}
    u(x) = \exp \left( -\frac{(x_1 - 0.5)^2 + (x_2 - 0.5)^2}{2 \sigma^2} \right) \cos(2\pi m x_1 ), \quad \sigma = 0.15, m = 8.
\end{equation}
For numerical integration, we divide the square domain $[0,1]^2$ into $50^2$ uniform square subdomains and use a Gaussian quadrature of order 3 in each subdomain. 
\begin{figure}
    \centering
    \includegraphics[width=0.4\textwidth]{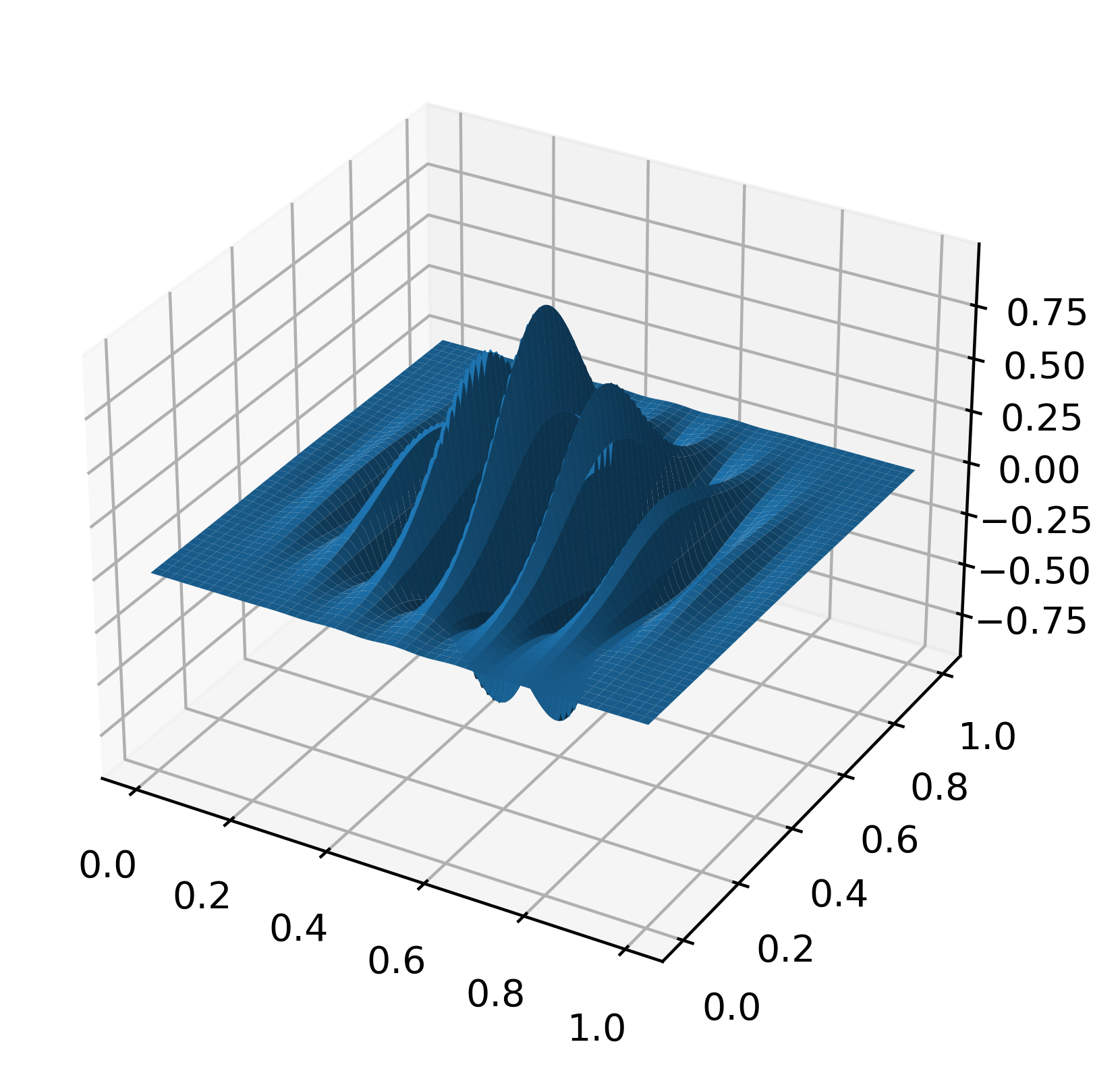}
    \caption{2D Gabor function \eqref{eq:2d-gabor}}
    \label{fig:2dgabor}
\end{figure}

In 3D, we fit the following function
\begin{equation}\label{eq:3d-sines}
u(x) = \sin(\pi x_1)\sin(\pi x_2) \sin(\pi x_3).
\end{equation}
For numerical integration, we divide the cubic domain $[0,1]^3$ into $25^3$ uniform cubic subdomains and use a Gaussian quadrature of order 3 in each subdomain. 

In 4D, we fit the following function
\begin{equation}\label{eq:4dsines}
    u(x) = \prod_{i =1}^4 \sin(\pi x_i),
\end{equation}
For the numerical integration, we use Quasi-Monte Carlo method~\cite{Caflisch:1998,MC:1995} with $5\times 10^5$ integration points generated by the Sobol sequence. 
For this task, we test shallow neural networks with $\operatorname{ReLU}^1$ and $\operatorname{ReLU}^4$ activation, respectively.

In 10D, we fit the Gaussian function  
\begin{equation}\label{eq:10gaussian}
    u(x) = \exp \left( -  \sum_{ i= 1}^d c(x_i - \omega)^2 \right), ~c = \frac{7.03}{d}, ~d = 10, ~\omega = 0.5. 
\end{equation}
For the numerical integration, we use Quasi Monte-Carlo method with $10^6$ integration points generated by the Sobol sequence. 

Results for 2D and 3D are shown in \cref{fig:123d-ex1-relu1}. 
The $L^2$ errors are plotted against the number of iterations, with the optimal convergence rate in \cref{Thm:WOGA-convex} plotted as a dotted reference line. 
For the 2D example, from \cref{fig:123d-ex1-relu1} (a), we see that using randomized dictionaries gives more accurate numerical results and achieves the optimal convergence order with a smaller discrete dictionary. 
On the other hand, when using a deterministic discrete dictionary, the accuracy deteriorates significantly as the number of iterations increases, failing to achieve an optimal convergence if the size of the deterministic dictionary is not larger than that of a randomized one. 
For the 3D example, from \cref{fig:123d-ex1-relu1} (b), we observe a much more significant improvement after using randomized dictionaries. 
We see that using randomized dictionaries of size only $2^{6}$ outperforms using a deterministic dictionary of size $2^{14}$ which is more than 200 times larger.  
The results for the 4D example are plotted in \cref{fig:4d-ex1-relu14}.
In both cases, we can see that using a randomized dictionary of a much smaller size leads to better accuracy than using a deterministic dictionary of a larger size. 
The reduction in the number of dictionary elements is of several orders of magnitude. 
Furthermore, the optimal convergence order is observed for using randomized dictionaries. 
The numerical results for the 10D example are plotted in Figure \ref{fig:10d-ex1}. 
We observe that using randomized dictionaries, we can still obtain very accurate numerical approximations that achieve an optimal convergence order in such a high-dimensional problem with a moderate dictionary size of around $2^9$. 
\begin{figure}[H] 
    \centering
    \subfloat[][]{\includegraphics[width=0.45\textwidth]{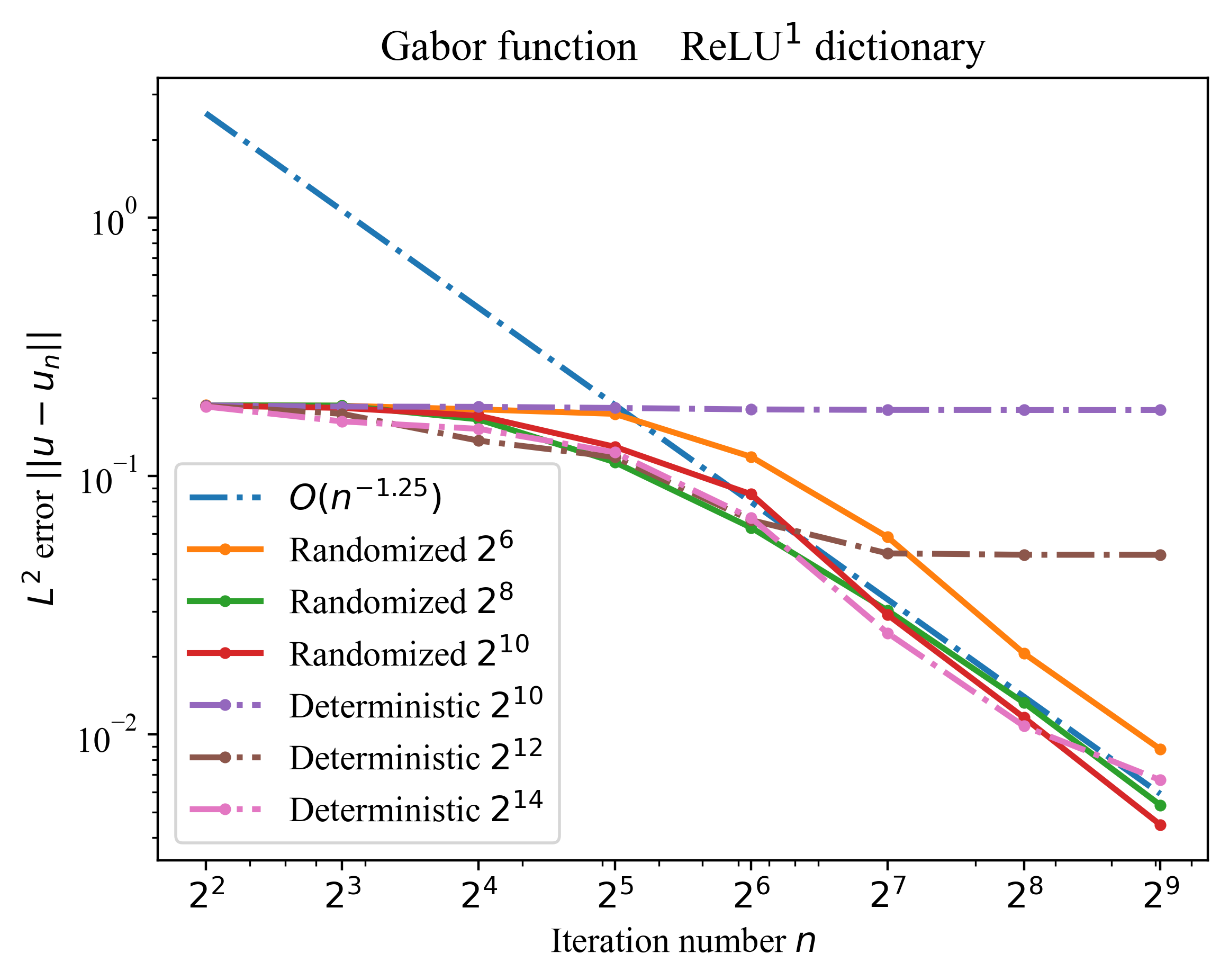}} 
    \subfloat[][]{\includegraphics[width=0.45\textwidth]{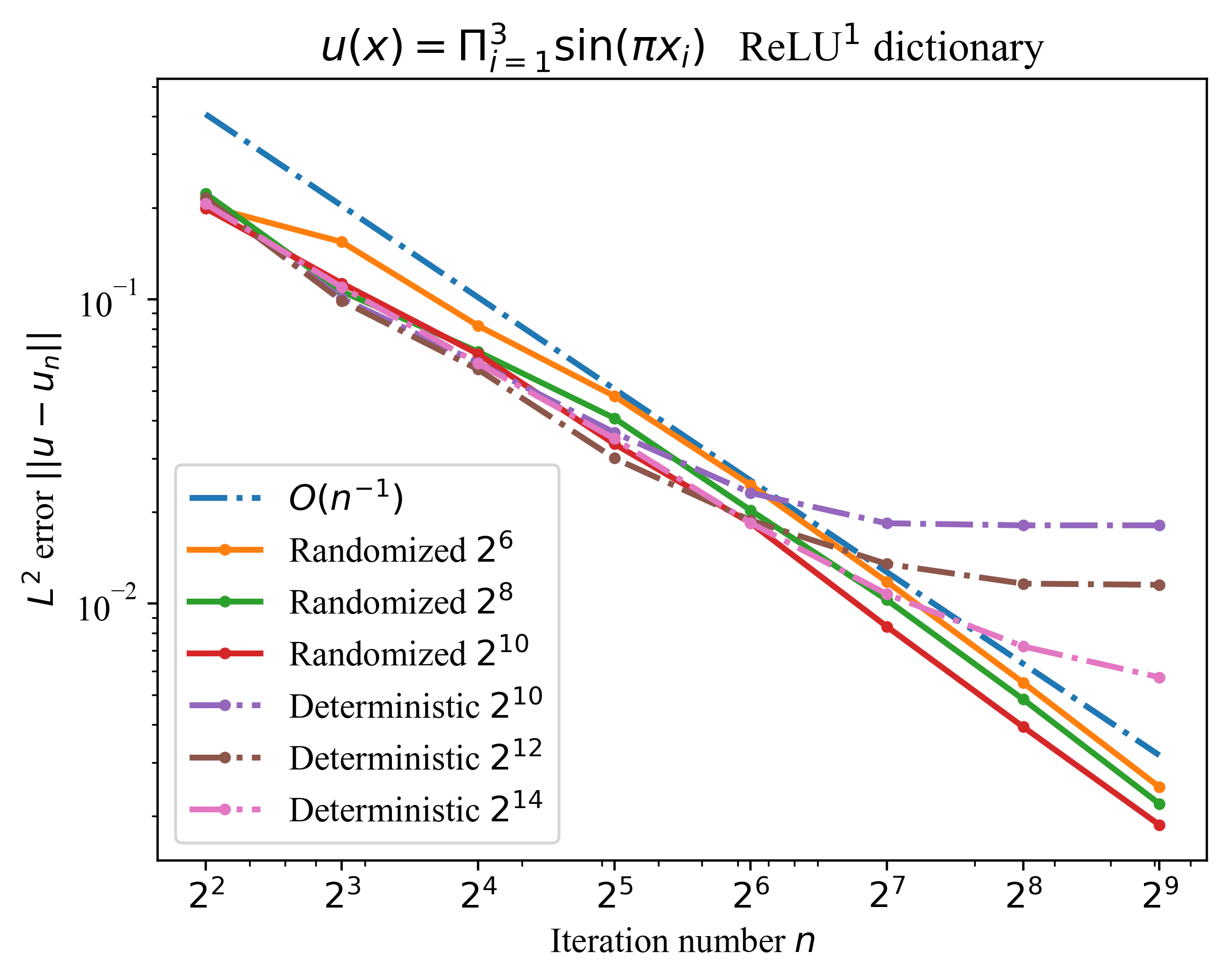}}
    \caption{Comparison between using 
 $\operatorname{ReLU}^1$ randomized dictionaries and $\operatorname{ReLU}^1$ deterministic dictionaries of various sizes. \textbf{(a)}: 2D $L^2$-minimization problem \eqref{eq:2d-gabor}. \textbf{(b)}: 3D $L^2$-minimization problem \eqref{eq:3d-sines}. }
    \label{fig:123d-ex1-relu1}
\end{figure}

\begin{figure}[H]
    \centering
    \subfloat[][]{\includegraphics[width=0.45\textwidth]{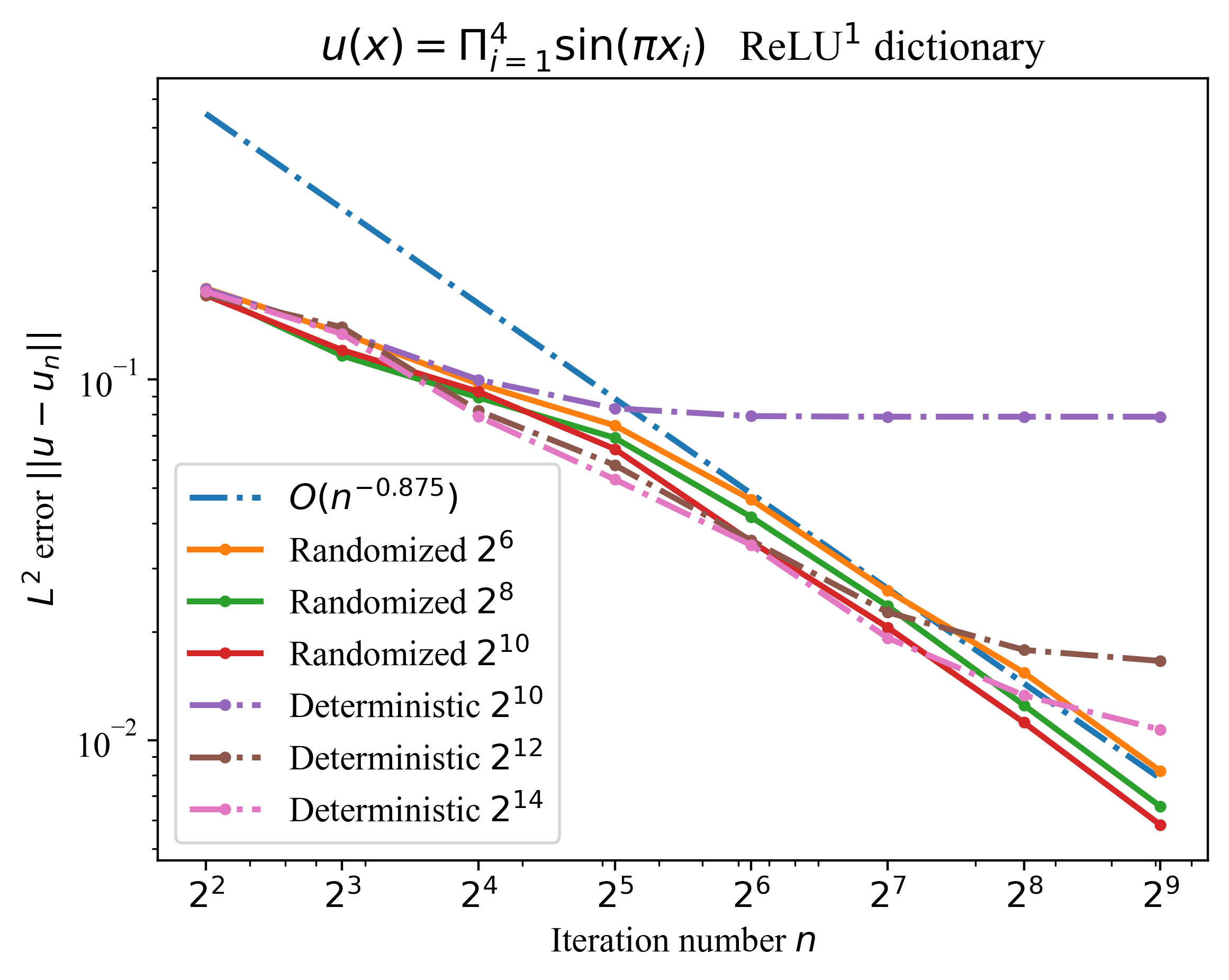}}
    \subfloat[][]{\includegraphics[width=0.45\textwidth]{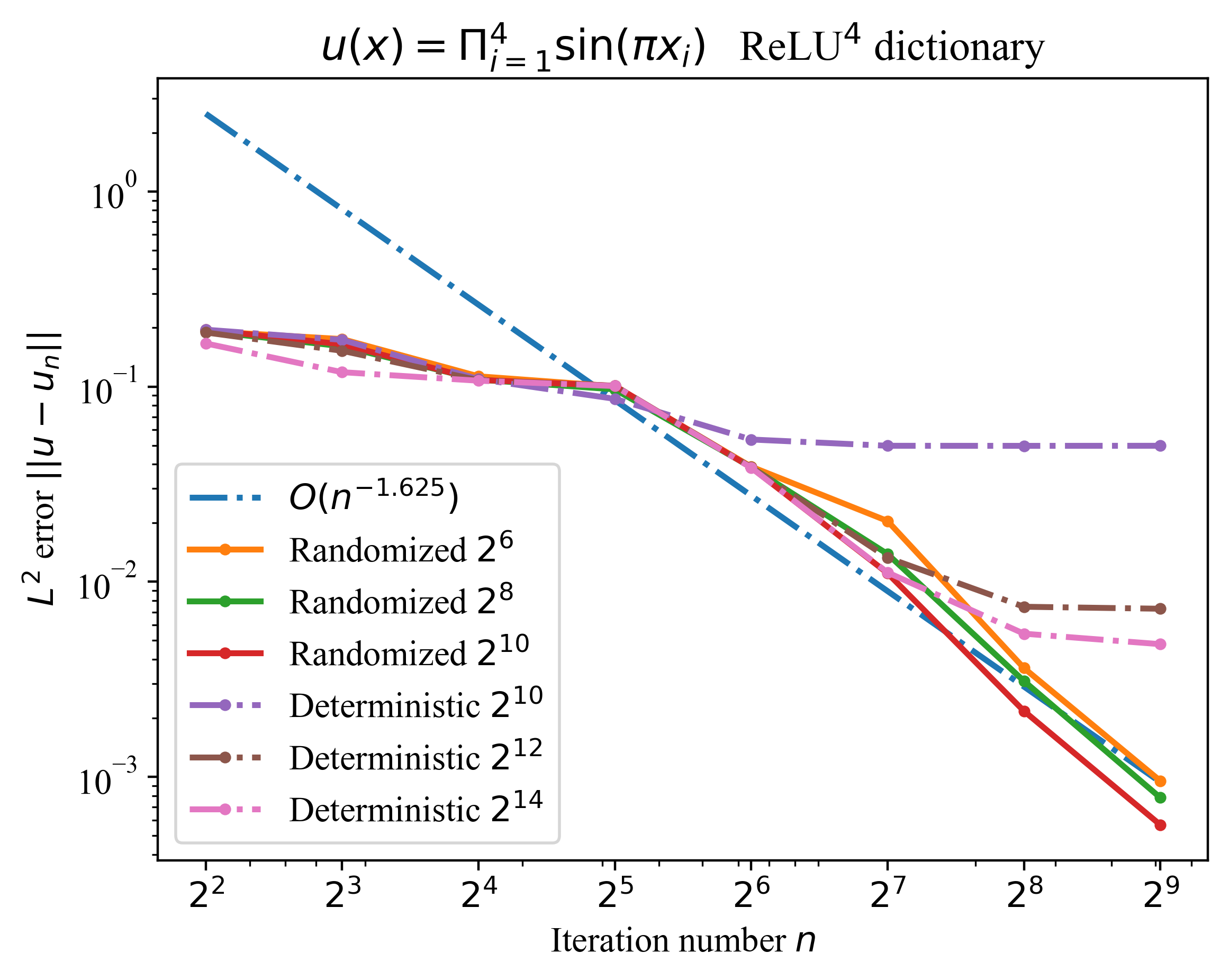}}
    \caption{Comparison between using randomized dictionaries and deterministic dictionaries of various sizes. 4D $L^2$-minimization problem \eqref{eq:4dsines}. \textbf{(a)}: $\operatorname{ReLU}^1$ activation. \textbf{(b)}: $\operatorname{ReLU}^4$ activation.}
    \label{fig:4d-ex1-relu14}
\end{figure}


\begin{figure}[H]
    \centering
     \subfloat[][]{\includegraphics[width=0.45\textwidth]{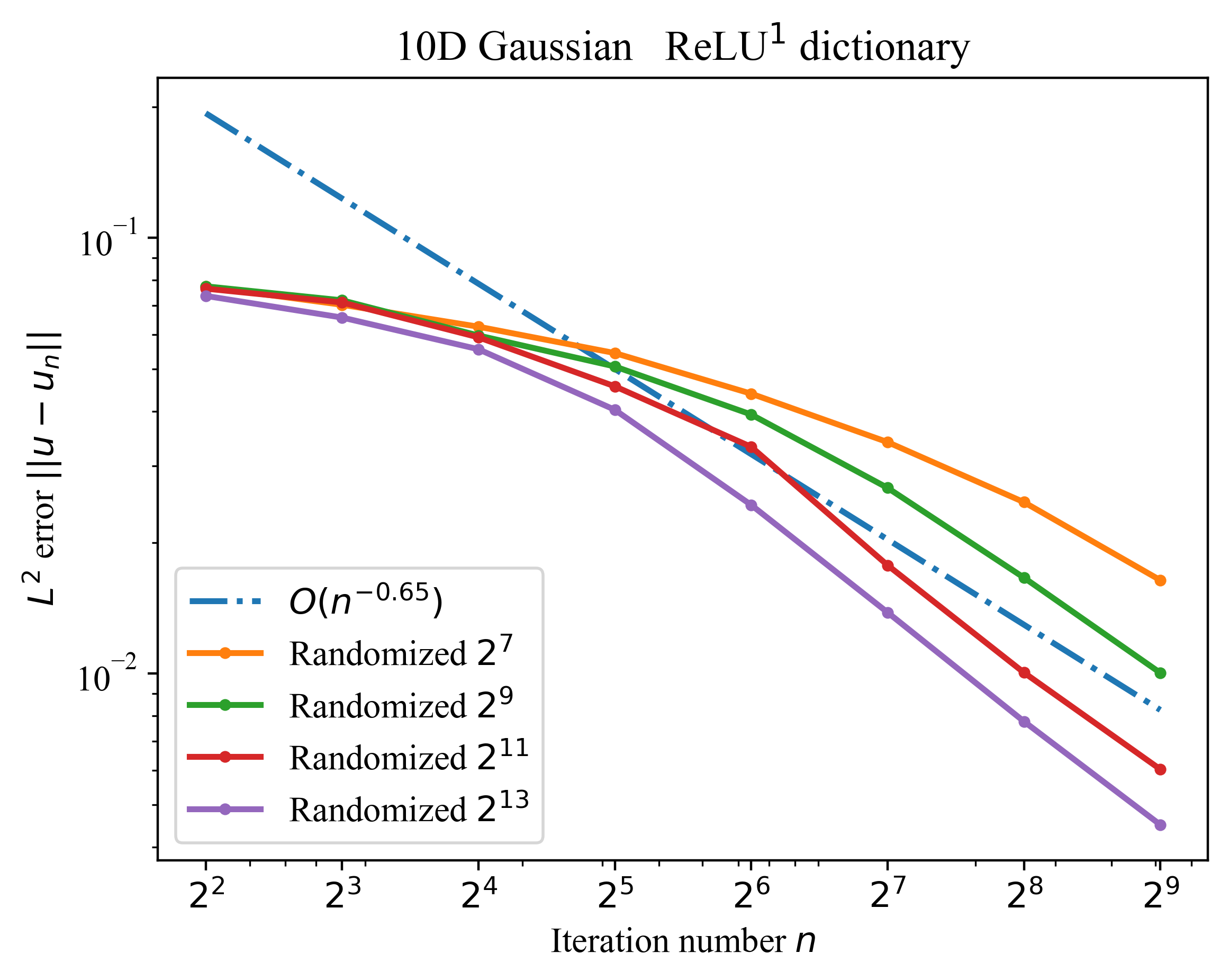}}
    \subfloat[][]{\includegraphics[width=0.45\textwidth]{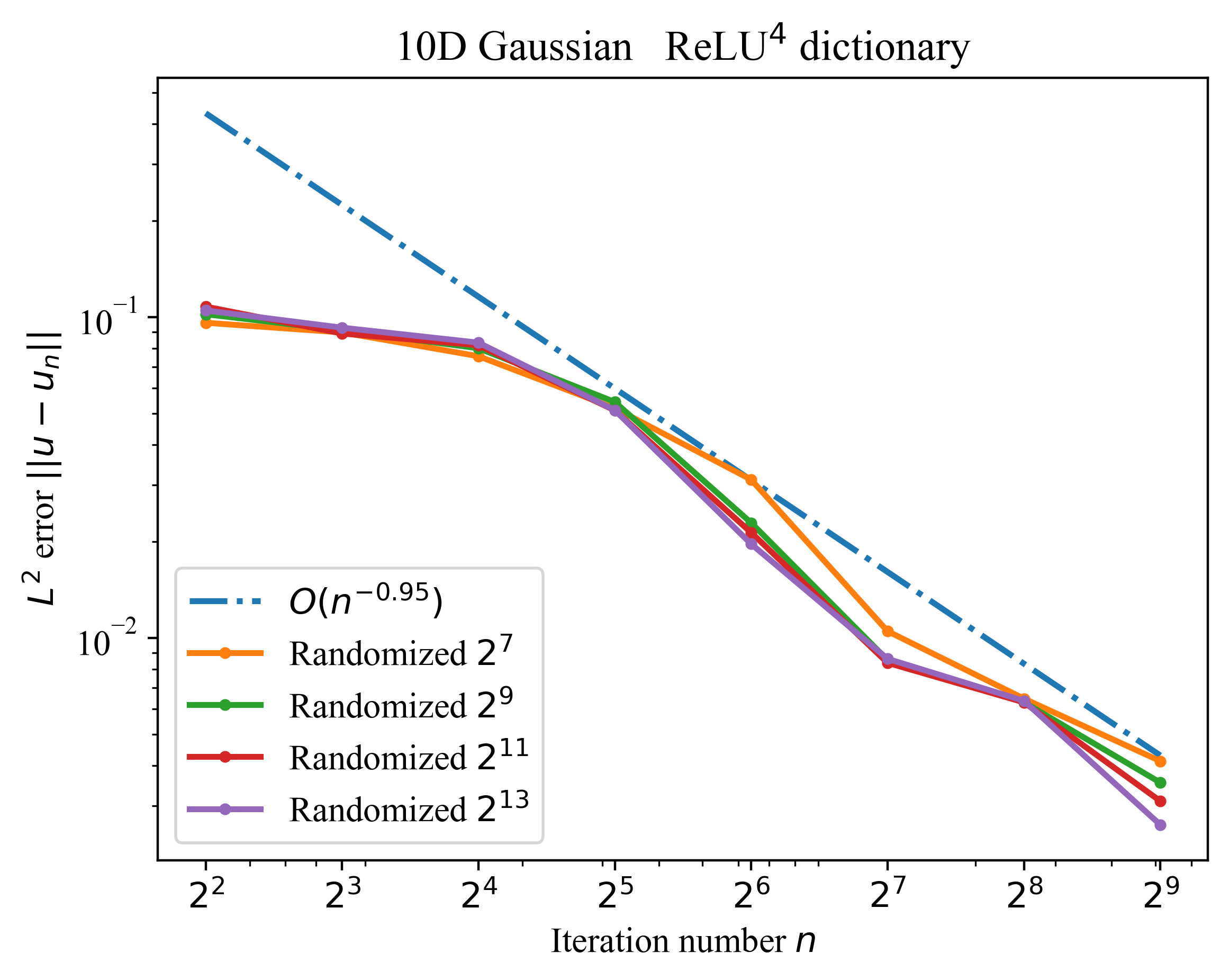}}
    \caption{10D $L^2$-minimization problem \eqref{eq:10gaussian} with randomized dictionaries. \textbf{(a)}: $\operatorname{ReLU}^1$ activation. 
    \textbf{(b)}: $\operatorname{ReLU}^4$ activation. }
    \label{fig:10d-ex1}
\end{figure}

\subsection{Linear Elliptic PDEs}
 We solve the following problem
\begin{equation}\label{eq:neumann-model}
    - \nabla \cdot (\alpha \nabla u) + u = f \quad \text{ in } \Omega, \quad
    \frac{\partial u}{\partial n} = g \quad \text{ on } \partial \Omega.
\end{equation}
Examples of solving this PDE in 1D and 2D with exact solutions $u(x_1) = \cos(\pi x_1)$ and $u(x_1,x_2) = \cos(\pi x_1)\cos(\pi x_2)$, respectively, are already demonstrated in \cite{SHJHX:2023} using OGA with a deterministic dictionary, while higher dimensional examples are not possible due to huge computational costs. 
Therefore, to showcase the effectiveness of OGA with randomized dictionaries, we carry out experiments in 3D and higher dimensions.  
We only carry out experiments with randomized dictionaries since comparisons with using deterministic dictionaries are already demonstrated in numerous $L^2$-minimization problems. 

We consider the following three-dimensional example on $\Omega = (0,1)^3$ for~\eqref{eq:neumann-model}:
\begin{equation}
    \label{eq:3d-neumann-1}
    \Omega = (0,1)^3, \quad
    \alpha = 1, \quad
    f(x) = (1 + 3\pi^2) \prod_{i = 1}^3\cos \pi x_i, \quad g(x) = 0,  
\end{equation}
whose exact solution is $u(x) = \prod_{i = 1}^3\cos \pi x_i$.

We also consider the following three-dimensional example on $\Omega = (0,1)^3$ for~\eqref{eq:neumann-model} with an oscillatory coefficient function $\alpha$:
\begin{equation}
    \label{eq:3d-neumann-2}
    \alpha(x) = \frac{1}{2}\sin(6\pi x_1) + 1, 
    \quad
    u(x) = \prod_{i = 1}^3\cos \pi x_i, \quad g(x)= 0.
\end{equation}
The corresponding right hand side $f(x)$ is computed. 

We then consider the following four-dimensional and ten-dimensional examples on $\Omega = (0,1)^d$ for~\eqref{eq:neumann-model} with a constant $\alpha$:
\begin{equation}
    \label{eq:4d-neumann-1}
    \alpha = 1, 
    \quad
    u(x) = \exp \left( -  \sum_{ i= 1}^d c(x_i - \omega)^2 \right), ~c = \frac{7.03}{d}, ~d = 4 \text{~or~} 10, ~\omega = 0.5.
\end{equation}
The right hand side $f(x)$ and the boundary condition function $g(x)$ are computed accordingly.

Regarding the numerical integration, for examples \eqref{eq:3d-neumann-1} and \eqref{eq:3d-neumann-2}, we divide the cubic domain into $50^3$ subdomains and use a Gaussian quadrature of order 3 in each subdomain.
For the example \eqref{eq:4d-neumann-1} with $d = 4$, within the domain $\Omega$, we use Quasi Monte-Carlo methods with $5\times 10^5$ integration points generated by the Sobol sequence.
On the boundary, we divide it into $50^3$ subdomains and use a Gaussian quadrature of order 3 in each subdomain.
For the example \eqref{eq:4d-neumann-1} with $d = 10$, we use $2\times 10^6$ integration points and $2\times 10^5$ integration points generated by the Sobol sequence, within the domain and on each side of the boundary, respectively. 
The size of a randomized dictionary used in the 3D and 4D examples is only $2^8$ and that used in the 10D example is $2^{11}$. 

We report the errors under the $L^2$ norm and the $H^1$ seminorm for each experiment in Tables \ref{tab:pde-ex1}, \ref{tab:pde-ex2}, \ref{tab:pde-ex3} and  \ref{tab:pde-ex4}. 
We observe that, in the 3D and 4D cases, the optimal convergence orders are achieved for both $L^2$ and $H^1$ errors. 
We also note that the actual convergence order is consistently slightly better than the optimal convergence order. 
For the 10D example, the overall convergence order is also better the optimal one, achieving good accuracy despite a drop in the order at the neuron number $256$. 

\begin{table}[H]
    \centering
    \begin{tabular}{ccccc}
    \hline 
        $n$ & $\|u - u_n \|_{L^2}$  &  order $(n^{-5/3})$& $|u - u_n |_{H^1}$ & order $(n^{-4/3})$ \\ \hline 
        16 & 9.56e-2 & * &1.54e0 & *\\
        32 & 2.34e-2 & 2.03 & 4.93e-1 &1.65 \\
        64 &5.07e-3&2.21 & 1.41e-1& 1.81 \\
        128 & 9.75e-4 	&2.38 &3.75e-2 & 1.91 \\
         256& 2.43e-4 & 2.01 & 1.26e-2& 1.58\\ 
          512& 6.75e-5 &1.84 & 4.55e-3 & 1.47 \\ \hline 
    \end{tabular}
    \caption{ $L^2$ and $H^1$ errors using $\operatorname{ReLU}^3$ activation for 3D Neumann problem with \eqref{eq:3d-neumann-1}.}
    \label{tab:pde-ex1}
\end{table}

\begin{table}[H]
    \centering
    \begin{tabular}{ccccc}
    \hline 
        $n$ & $\|u - u_n \|_{L^2}$  &  order $(n^{-5/3})$& $|u - u_n |_{H^1}$ & order $(n^{-4/3})$\\ \hline 
        16 &9.56e-2 & * & 1.54e0 & *  \\
        32 &3.31e-2 &1.53 & 6.70e-1 & 1.20 \\
        64 & 4.86e-3 & 2.77& 1.32e-1 & 2.35 \\
        128 & 1.11e-3& 2.13& 4.2e-2 &1.65 \\
         256&2.65e-4 & 2.07 &  1.34e-2& 1.65\\ 
          512& 6.49e-5&2.03 & 4.32e-3 & 1.64 \\ \hline 
    \end{tabular}
    \caption{$L^2$ and $H^1$ errors using $\operatorname{ReLU}^3$ activation for 3D Neumann problem with \eqref{eq:3d-neumann-2}.}
    \label{tab:pde-ex2}
\end{table}

\begin{table}[H]
    \centering
    \begin{tabular}{ccccc}
    \hline 
        $n$ & $\|u - u_n \|_{L^2}$  &  order $O(n^{-1.625})$& $|u - u_n |_{H^1}$ & order $O(n^{-1.375})$\\ \hline 
        16 & 7.49e-2 & * &1.16e0 & *\\
        32 & 5.31e-2 & 0.50 &1.03e0 & 0.173 \\
        64 & 1.08e-2&2.30  & 2.64e-1 & 1.96 \\
        128 & 3.81e-3& 1.50 & 1.02e-1 &  1.36\\
         256& 4.91e-4 & 2.95 & 1.69e-2& 2.60\\ 
          512& 1.32e-4 & 1.90 & 5.60e-3& 1.59 \\ \hline 
    \end{tabular}
    \caption{$L^2$ and $H^1$ errors using $\operatorname{ReLU}^4$ activation for 4D Neumann problem with \eqref{eq:4d-neumann-1}.}
    \label{tab:pde-ex3}
\end{table}

\begin{table}[H]
    \centering
    \begin{tabular}{ccccc}
    \hline 
        $n$ & $\|u - u_n \|_{L^2}$  &  order $O(n^{-0.95})$& $|u - u_n |_{H^1}$ & order $O(n^{-0.85})$\\ \hline 
        16& 1.01e-1 & * &  1.72e0 &    * \\
        32 & 6.69e-2 & 0.59 &  1.19e0  &   0.53\\
        64 & 2.16e-2 &1.63  &  4.73e-1 &    1.33 \\
        128 & 8.79e-3 &   1.29 &  2.53e-1  & 0.90 \\
         256& 6.58e-3 & 0.42 &    2.16e-1  & 0.23 \\ 
          512& 3.28e-3 &1.00 &   9.08e-2  &  1.25 \\ \hline 
    \end{tabular}
    \caption{$L^2$ and $H^1$ errors using $\operatorname{ReLU}^4$ activation for 10D Neumann problem with \eqref{eq:4d-neumann-1}.}
    \label{tab:pde-ex4}
\end{table}


\begin{subsection}{Nonlinear elliptic PDEs}
We consider the nonlinear Poisson--Boltzmann equation
\begin{equation} \begin{split}
\label{eq:PB}
- \Delta u + \sinh (\alpha u) = f \quad &\text{ in } \Omega, \\
\frac{\partial u}{\partial n} = g \quad &\text{ on } \partial \Omega,
\end{split} \end{equation}
where $\alpha > 0$.

For the numerical experiments, we set $\alpha = 1$, $\Omega = (0,1)^d$ in \eqref{eq:PB}. 
we use the following Gabor function as the exact solution
\begin{equation}
    \label{eq:gabor}
    u(x)  = \exp(-\frac{\sum_{i = 1}^d (x_i - 0.5)^2}{2 \sigma^2}) \cos(2\pi m x_1 ),  
\end{equation}
where $\sigma = 0.15$ and $ m = 4$, 
and compute the corresponding right hand side function $f$ and the boundary condition $g$.

We also carry out numerical experiments in dimensions 1, 2, 3.
We use the $ReLU^k$ shallow neural network with $k = 3$. 
For numerical integration, in 1D, we divide the interval into $2^{14}$ uniform subintervals and use a Guassian quadrature of order 3 on each subinterval. 
In 2D, we divide the square domain into $600^2$ square subdomains and use a Gaussian quadrature of order 2.
In 3D, we divide the cubic domain into $50^3$ cubic subdomains and use a Gaussian quadrature of order 2
In all these cases, we use a randomized dictionary of size dictionary size 512.
The results are shown in Tables \ref{tab:pde2-1d-ex1},\ref{tab:pde2-2d-ex1}, and \ref{tab:pde2-3d-ex1}.

\begin{table}[H]
    \centering
    \begin{tabular}{ccccc}
    \hline 
$n$  & 	 $\|u-u_n \|_{L^2}$ & 	 order $O(n^{-4})$& 	 $ | u -u_n |_{H^1}$ & 	 order $O(n^{-3})$ \\ \hline \hline 
8 		 &  3.43e-01 &  		* &  		 7.60e+00 &  		 * \\ \hline  
16 		 &  8.89e-03 &  		 5.27 &  		 5.21e-01 &  		 3.87 \\ \hline  
32 		 &  4.06e-04 &  		 4.45 &  		 4.94e-02 &  		 3.40 \\ \hline  
64 		 &  2.26e-05 &  		 4.17 &  		 6.11e-03 &  		 3.02 \\ \hline  
128 		 &  1.07e-06 &  		 4.40 &  		 6.25e-04 &  		 3.29 \\ \hline  
256 		 &  6.37e-08 &  		 4.07 &  		 7.64e-05 &  		 3.03 \\ \hline 
    \end{tabular}
    \caption{ 1D: $L^2$ and $H^1$ errors using $\operatorname{ReLU}^3$ activation for \eqref{eq:PB} with exact solution \eqref{eq:gabor}}
    \label{tab:pde2-1d-ex1}
\end{table}

\begin{table}[H]
    \centering
    \begin{tabular}{ccccc}
    \hline 
$n$  & 	 $\|u-u_n \|_{L^2}$ & 	 order $O(n^{-2.25})$ & 	 $ | u -u_n |_{H^1}$ & 	 order $O(n^{-1.75})$ \\ \hline \hline 
16 		 &  1.88e-01 &  		* &  		 5.65e+00 &  		 * \\ \hline 
32 		 &  1.39e-01 &  		 0.44 &  		 4.58e+00 &  		 0.30 \\ \hline  
64 		 &  3.64e-02 &  		 1.93 &  		 1.60e+00 &  		 1.52 \\ \hline  
128 		 &  6.12e-03 &  		 2.58 &  		 4.23e-01 &  		 1.92 \\ \hline  
256 		 &  8.37e-04 &  		 2.87 &  		 9.24e-02 &  		 2.19 \\ \hline  
512 		 &  1.42e-04 &  		 2.56 &  		 2.33e-02 &  		 1.99 \\ \hline 
    \end{tabular}
    \caption{2D: $L^2$ and $H^1$ errors using $\operatorname{ReLU}^3$ activation for \eqref{eq:PB} with exact solution \eqref{eq:gabor}}
    \label{tab:pde2-2d-ex1}
\end{table}

\begin{table}[H]
    \centering
    \begin{tabular}{ccccc}
    \hline 
$n$  & 	 $\|u-u_n \|_{L^2}$ & 	 order $O(n^{-1.67})$ & 	 $ | u -u_n |_{H^1}$ & 	 order $O(n^{-1.33})$ \\ \hline \hline 
32 		 &  9.69e-02 &  	* &  	 3.38e+00 &  		 * \\ \hline  
64 		 &  8.72e-02 &  		 0.15 &  		 3.15e+00 &  		 0.10 \\ \hline  
128 		 &  5.52e-02 &  		 0.66 &  		 2.28e+00 &  		 0.46 \\ \hline  
256 		 &  2.25e-02 &  		 1.29 &  		 1.15e+00 &  		 1.00 \\ \hline  
512 		 &  5.47e-03 &  		 2.04 &  		 3.83e-01 &  		 1.58 \\ \hline  
1024 		 &  1.36e-03 &  		 2.01 &  		 1.29e-01 &  		 1.57 \\ \hline 
    \end{tabular}
    \caption{3D: $L^2$ and $H^1$ errors using $\operatorname{ReLU}^3$ activation for \eqref{eq:PB} with exact solution \eqref{eq:gabor}}
    \label{tab:pde2-3d-ex1}
\end{table}

\end{subsection}

\begin{subsection}{General Linear Elliptic PDEs}
To further demonstrate the broad application of the randomized OGA, we solve a more general elliptic PDE of the form 
\begin{equation}
    \begin{aligned}
    - \nabla \cdot \left( a(x) \nabla u \right) + \textbf{b}(x)\cdot \nabla u + c(x)u & = f(x), \quad \text{in} \quad \Omega = (0,1)^d,\\
    \frac{\partial u}{\partial n} & = g, \quad \text{on} \quad \Omega. 
    \end{aligned}
\end{equation}
We remark that, due to the indefiniteness of the PDE, the convergence theorem that we established for OGA does not apply. 
However, we can still apply the randomized OGA (\cref{alg:OGA-randomized-dictionary}) to solve these equations.
To this purpose, we made a slight modification to the argmax subproblem (See \ref{sec:oga-general-elliptic}). 
In addition, as mentioned in \cref{Rem:other-activation}, we can also use other activation functions for the randomized orthogonal greedy algorithm. 
For the following experiments, we test both the $\text{ReLU}^k$ and the $\tanh$ activation functions. 
We consider the following setting for numerical experiments: 
\begin{equation}\label{eq:3d-general-1}
   d = 3, a(x) \equiv 1, \textbf{b}(x) \equiv (5,5,5)^T, c(x)\equiv -4, u(x) = \prod_{i=1}^3 \cos(2\pi x_i).
\end{equation}
The right-hand side $f(x)$ is computed correspondingly.
Regarding numerical integration, we divide the cubic domain into $50^3$ uniform subdomains and use a Gaussian quadrature of order 3 in each subdomain.

We report the errors under the $L^2$ norm and the $H^1$ seminorm for each experiment in Table \ref{tab:pde-ex5}.
We also test using the tanh activation function for the shallow neural network. 
The errors are reported in Table \ref{tab:pde-ex6}. 
Parameters of the randomized discrete dictionary are sampled from a uniform distribution on $[-R_m, R_m]^{d+1}$, where $R_m$ needs to be manually tuned and is set to 0.4. 
For both $\text{ReLU}^3$ and $\tanh$ activations, we observe very good convergence. 
In particular, we see that using the $tanh$ activation function yields better convergence results compared to using $\text{ReLU}^3$ activation. However, the range of the uniform distribution in this case is not known a prior and is determined through trial and error. 
Finally, we remark that, recently, there have been other numerical evidence showing that OGA works for indefinite problems \cite{Hong:2024}. 
\begin{table}[H]
    \centering
    \begin{tabular}{ccccc}
    \hline 
        $n$ & $\|u - u_n \|_{L^2}$  &  order & $|u - u_n |_{H^1}$ & order \\ \hline 
			16 		 &  3.98e-01 &  		 *  &  		 6.67e+00 &  * \\ \hline
			32 &  2.38e+00 &  		 -2.58 &   6.63e+00 &  		 0.01 \\ \hline  
		64  &  1.09e+00 &  		 1.12 &  2.72e+00 &  1.28 \\ \hline  	
    128  &  1.30e-01 &  3.07 & 6.81e-01 &  2.00 \\ \hline  	
    256  &  7.09e-03 &  4.20 &  2.166e-01 &  1.65 \\ \hline  
    512 &  2.65e-03 &  	1.42 &  1.086e-01 &  1.00 \\ \hline 
    \end{tabular}
    \caption{$L^2$ and $H^1$ errors using $\operatorname{ReLU}^3$ activation for a 3D general elliptic PDE with \eqref{eq:3d-general-1}.}
    \label{tab:pde-ex5}
\end{table}
\begin{table}[H]
    \centering
    \begin{tabular}{ccccc}
    \hline 
        $n$ & $\|u - u_n \|_{L^2}$  &  order & $|u - u_n |_{H^1}$ & order \\ \hline 
16 		 &  7.38e-01 &  * &  		 6.53e+00 &  		* \\ \hline  
32 		 &  8.83e+00 &  		 -3.58 &  		 1.01e+01 &  		 -0.63 \\ \hline  
64 		 &  9.09e-01 &  		 3.28 &  	 2.84e+00 &  		 1.83 \\ \hline  
128 		 &  6.63e-02 &  		 3.78 &  		 1.14e+00 &  		 1.31 \\ \hline  
256 		 &  4.88e-03 &  		 3.76 &  		 1.91e-01 &  		 2.58 \\ \hline  
512 		 &  4.05e-04 &  		 3.59 &  		 2.12e-02 &  		 3.17 \\ \hline 
    \end{tabular}
    \caption{$L^2$ and $H^1$ errors using $\tanh$ activation for a 3D general elliptic PDE with \eqref{eq:3d-general-1}. $R_m =0.4$. }
    \label{tab:pde-ex6}
\end{table}

\end{subsection}

\subsection{Fixed-Size Iterative OGA}

The randomized OGA builds a shallow neural network incrementally by adding one neuron at a time.
Here, we introduce a variant called the Fixed-Size Iterative OGA (\cref{alg:iterative_OGA}), which differs from standard OGA in that it maintains a fixed number of neurons.
At each iteration, one neuron is removed and replaced with a newly selected neuron, allowing the network to explore more favorable representations without changing its overall size. Notably, we demonstrate that this method can significantly enhance solutions obtained using the popular random feature method \cite{CCEY:2022}.

\begin{algorithm}[H]
    \caption{Fixed-Size Iterative Orthogonal Greedy Algorithm (Fixed-Size Iterative OGA)}
    \label{alg:iterative_OGA}
    \begin{algorithmic}
        \State \textbf{Initialize:} A shallow neural network $u_n(x) = \sum_{j=1}^n a_j g_j(x)$, $g_j \in \mathbb{D}$ for $j=1,\dots,n$. Define the subspace $H_n = \text{span}\{g_j\}_{j=1}^n$.
        \For{$k = 1, \dots, K$}
            \State $i = (k-1)\mod n + 1 $ \Comment{Cycle through neuron indices}
            \State \textbf{Neuron removal:} Remove $g_i$ to form $\check{H}_{n-1} := \text{span}\{g_j\}_{j \neq i}$. Solve:
            \begin{equation*}
                u_{n-1} = \operatorname*{arg\,min}_{v \in \check{H}_{n-1}} E(v).
            \end{equation*}
            \State \textbf{Generate Randomized Dictionary:} Sample $(\phi_{i,k}, b_{i,k}) \sim \text{Uniform}(R)$ for $i = 1, \dots, N$ to construct:
            \begin{equation*}
                \mathbb{D}_{N,k} = \left\{ \pm \sigma_k(\mathcal{S}(\phi_{i,k}) \cdot x + b_{i,k}) : 1 \leq i \leq N \right\}.
            \end{equation*}
            \State \textbf{Neuron Selection:} Find a new neuron $g_i$:
            \begin{equation*}
                g_i = \operatorname*{arg\,max}_{g \in \mathbb{D}_{N,k}} \langle \nabla E(u_{n-1}), g \rangle.
            \end{equation*}
            \State \textbf{Form New Subspace:} Update the subspace $H_n = \text{span}\{g_j\}_{j=1}^n$ and solve:
            \begin{equation*}
                u_n = \operatorname*{arg\,min}_{v \in H_n} E(v).
            \end{equation*}
        \EndFor
        \State \textbf{Return:} $u_n$.
    \end{algorithmic}
\end{algorithm}

We demonstrate the effectiveness of our Fixed-Size Iterative OGA using an $L^2$-minimization example in 2D. 
We use a ReLU shallow neural network to approximate the target function $u(x_1,x_2) = \cos(4\pi x_1)\cos(2\pi x_2)$.
For the Fixed-Size Iterative OGA, we set the total number of iterations $ K = 2n$.
To initialize the neural network, we apply the random feature method: we first sample the parameters $\{(\omega_i, b_i)\}_{i=1}^n$ from a uniform distribution on \(S^{d-1}\times [c_1, c_2]\), then solve a least squares problem for the coefficients $\{a_i\}_{i=1}^n$.
We run both the Fixed-Size Iterative OGA and the random feature method for 10 independent trials and report the averaged $L^2$ errors.
For numerical integration, the domain $[0,1]^2$ is divided into $100^2$ uniform subdomains, and we use a Gaussian quadrature of order 3 in each subdomain. 
The errors are presented in Table~\ref{tab:iterative_OGA}.
Although our approach incurs a higher computational cost, it significantly improves the solution obtained from the random feature method.

\begin{table}[H]
	\centering
	\begin{tabular}{c|cc|cc}
		\multicolumn{1}{c}{} & \multicolumn{2}{c|}{\textbf{Random feature method}} & \multicolumn{2}{c}{\textbf{Fixed-Size Iterative OGA}}  \\ \hline \hline
		\textbf{$n$} & $\|u - u_n\|_{L^2}$ & order & $\|u - u_n\|_{L^2}$ & order  \\ \hline
		32  & 4.54e-01 &  * &  7.64e-02 &  *   \\ \hline
		64   &2.80e-01 & 0.70 & 2.97e-02 &  	1.36  \\ \hline
		128 & 1.89e-01 &  0.57& 1.25e-02 &  1.25  \\ \hline
        256 & 1.07e-01 &  		 0.82 & 4.91e-03 &  1.34 \\ \hline 
          512 & 3.79e-02 &  		 1.49 & 1.92e-03 &  1.35   \\ \hline 
	\end{tabular}
	\caption{$L^2$ errors: random feature method, Fixed-Size Iterative OGA. The reported errors are averaged over 10 trials.}
     \label{tab:iterative_OGA}
\end{table}

\section{Conclusion}\label{sec:conclusion}

OGA is proven to achieve an optimal convergence rate for approximation problems using shallow neural networks. In this paper, we extend this result by establishing the optimal convergence rate for convex optimization problems, thereby broadening its applicability to nonlinear PDEs.
Regarding the practical application of greedy algorithms, we discuss two approaches to discretize the dictionary: one using a deterministic hyperrectangular grid-based discretization and the other using a random discretization. We use the convergence theory of weak greedy algorithms as the primary tool to analyze these approaches, proving rigorous convergence results for practical greedy algorithms using discrete dictionaries.
Specifically, we show that both the deterministic discretization approach and the randomized one amount to a realization of a weak greedy algorithm, and we provide a quantitative analysis of the discrete dictionary size that is required to achieve a desired convergence order.
Finally, we carry out extensive numerical experiments on solving $L^2$ function approximation problems, linear and nonlinear elliptic PDEs to validate our theoretical convergence analysis and demonstrate the effectiveness of using randomized dictionaries for greedy algorithms in application to neural network optimization problems. 
In addition, we introduce a variant of the randomized OGA called Fixed-Size Iterative OGA, allowing for the training of shallow neural networks with a fixed number of neurons.

\section*{Reproducibility}
The code and data for reproducing the results in the paper are available to download from the link: \url{https://github.com/georgexxu/RandomizedOGA} 
\section*{Acknowledgement}
 The authors would like to thank Dr. Jongho Park for helpful discussions on the paper.

\section*{Declarations}
This work is supported by KAUST Baseline Research Fund. The authors declare no conflict of interest.

 \bibliographystyle{elsarticle-num} 
\bibliography{reference}

\begin{thebibliography}{10}
\expandafter\ifx\csname url\endcsname\relax
  \def\url#1{\texttt{#1}}\fi
\expandafter\ifx\csname urlprefix\endcsname\relax\def\urlprefix{URL }\fi
\expandafter\ifx\csname href\endcsname\relax
  \def\href#1#2{#2} \def\path#1{#1}\fi

\bibitem{Cuomo:2022}
S.~Cuomo, V.~S. Di~Cola, F.~Giampaolo, G.~Rozza, M.~Raissi, F.~Piccialli,
  Scientific machine learning through physics-informed neural networks: {W}here
  we are and what’s next, J. Sci. Comput. 92~(3) (2022) Paper No.~88.

\bibitem{RPK:2019}
M.~Raissi, P.~Perdikaris, G.~E. Karniadakis, Physics-informed neural networks:
  {A} deep learning framework for solving forward and inverse problems
  involving nonlinear partial differential equations, J. Comput. Phys. 378
  (2019) 686--707.

\bibitem{EY:2018}
W.~E, B.~Yu, The deep {R}itz method: a deep learning-based numerical algorithm
  for solving variational problems, Commun. Math. Stat. 6~(1) (2018) 1--12.

\bibitem{Xu:2020}
J.~Xu, Finite neuron method and convergence analysis, Commun. Comput. Phys.
  28~(5) (2020) 1707--1745.

\bibitem{HMX:2023}
J.~He, T.~Mao, J.~Xu, Expressivity and approximation properties of deep neural
  networks with ${R}e{LU}^k$ activation, arXiv preprint arXiv:2312.16483
  (2023).

\bibitem{HX:2023}
J.~He, J.~Xu, Deep neural networks and finite elements of any order on
  arbitrary dimensions, arXiv preprint arXiv:2312.14276 (2023).

\bibitem{HSTX:2022}
Q.~Hong, J.~W. Siegel, Q.~Tan, J.~Xu, On the activation function dependence of
  the spectral bias of neural networks, arXiv preprint arXiv:2208.04924 (2022).

\bibitem{GBC:2016}
I.~Goodfellow, Y.~Bengio, A.~Courville, Deep Learning, MIT Press, Cambridge,
  MA, 2016, \url{http://www.deeplearningbook.org}.

\bibitem{CGPPT:2020}
E.~C. Cyr, M.~A. Gulian, R.~G. Patel, M.~Perego, N.~A. Trask, Robust training
  and initialization of deep neural networks: {A}n adaptive basis viewpoint,
  in: Proceedings of The First Mathematical and Scientific Machine Learning
  Conference, Vol. 107 of Proceedings of Machine Learning Research, PMLR, 2020,
  pp. 512--536.

\bibitem{AS:2021}
M.~Ainsworth, Y.~Shin, Plateau phenomenon in gradient descent training of
  {R}e{LU} networks: Explanation, quantification, and avoidance, SIAM J. Sci.
  Comput. 43~(5) (2021) A3438--A3468.

\bibitem{AS:2022}
M.~Ainsworth, Y.~Shin, {A}ctive {N}euron {L}east {S}quares: {A} training method
  for multivariate rectified neural networks, SIAM J. Sci. Comput. 44~(4)
  (2022) A2253--A2275.

\bibitem{PXX:2022}
J.~Park, J.~Xu, X.~Xu, A neuron-wise subspace correction method for the finite
  neuron method, To appear in J. Comput. Math., arXiv preprint arXiv:2211.12031
  (2022).

\bibitem{CCEY:2022}
J.~Chen, X.~Chi, W.~E, Z.~Yang, Bridging traditional and machine learning-based
  algorithms for solving {PDE}s: {T}he random feature method, J. Mach. Learn. 1
  (2022) 268--298.

\bibitem{EMW:2020}
W.~E, C.~Ma, L.~Wu, A comparative analysis of optimization and generalization
  properties of two-layer neural network and random feature models under
  gradient descent dynamics, Sci. China Math. 63~(7) (2020) 1235--1258.

\bibitem{LL:2024}
A.~Lamperski, T.~Lekang, Approximation with random shallow {R}e{LU} networks
  with applications to model reference adaptive control, arXiv preprint
  arXiv:2403.17142 (2024).

\bibitem{RR:2007}
A.~Rahimi, B.~Recht, Random features for large-scale kernel machines, in:
  Advances in Neural Information Processing Systems, Vol.~20, Curran
  Associates, Inc., 2007.

\bibitem{DL:2021}
S.~Dong, Z.~Li, Local extreme learning machines and domain decomposition for
  solving linear and nonlinear partial differential equations, Comput. Methods
  Appl. Mech. Engrg. 387 (2021) Paper No.~114129.

\bibitem{HZS:2006}
G.-B. Huang, Q.-Y. Zhu, C.-K. Siew, Extreme learning machine: {T}heory and
  applications, Neurocomputing 70~(1-3) (2006) 489--501.

\bibitem{LLR:2024}
C.-O. Lee, Y.~Lee, B.~Ryoo, A nonoverlapping domain decomposition method for
  extreme learning machines: {E}lliptic problems, arXiv preprint
  arXiv:2406.15959 (2024).

\bibitem{DT:1996}
R.~A. DeVore, V.~N. Temlyakov, Some remarks on greedy algorithms, Adv. Comput.
  Math. 5~(1) (1996) 173--187.

\bibitem{Jones:1992}
L.~K. Jones, A simple lemma on greedy approximation in {H}ilbert space and
  convergence rates for projection pursuit regression and neural network
  training, Ann. Statist. 20~(1) (1992) 608--613.

\bibitem{MZ:1993}
S.~G. Mallat, Z.~Zhang, Matching pursuits with time-frequency dictionaries,
  IEEE Trans. Signal Process. 41~(12) (1993) 3397--3415.

\bibitem{Temlyakov:2011}
V.~Temlyakov, Greedy Approximation, Cambridge University Press, Cambridge,
  2011.

\bibitem{DPW:2021}
A.~Dereventsov, A.~Petrosyan, C.~Webster, Greedy shallow networks: {A}n
  approach for constructing and training neural networks, Int. J. Artif.
  Intell. Res. 19~(2) (2021).

\bibitem{SHJHX:2023}
J.~W. Siegel, Q.~Hong, X.~Jin, W.~Hao, J.~Xu, Greedy training algorithms for
  neural networks and applications to {PDE}s, J. Comput. Phys. 484 (2023) Paper
  No.~112084.

\bibitem{SX:2022a}
J.~W. Siegel, J.~Xu, Optimal convergence rates for the orthogonal greedy
  algorithm, IEEE Trans. Inform. Theory 68~(5) (2022) 3354--3361.

\bibitem{SX:2024}
J.~W. Siegel, J.~Xu, Sharp bounds on the approximation rates, metric entropy,
  and $n$-widths of shallow neural networks, Found. Comput. Math. 24 (2024)
  481--537.

\bibitem{BM:2002}
P.~L. Bartlett, S.~Mendelson, {R}ademacher and {G}aussian complexities: {R}isk
  bounds and structural results, J. Mach. Learn. Res. 3~(Nov) (2002) 463--482.

\bibitem{Temlyakov:2000}
V.~N. Temlyakov, Weak greedy algorithms, Adv. Comput. Math. 12~(2) (2000)
  213--227.

\bibitem{JLS:2024}
J.~Jia, Y.~J. Lee, R.~Shan, An unconstrained formulation of some constrained
  partial differential equations and its application to finite neuron methods,
  arXiv preprint arXiv:2405.16894 (2024).

\bibitem{SX:2023}
J.~W. Siegel, J.~Xu, Characterization of the variation spaces corresponding to
  shallow neural networks, Constr. Approx. 57~(3) (2023) 1109--1132.

\bibitem{KS:2023}
J.~M. Klusowski, J.~W. Siegel, Sharp convergence rates for matching pursuit,
  arXiv preprint arXiv:2307.07679 (2023).

\bibitem{BCDD:2008}
A.~R. Barron, A.~Cohen, W.~Dahmen, R.~A. DeVore, Approximation and learning by
  greedy algorithms, Ann. Statist. 36~(1) (2008) 64--94.

\bibitem{Zhang:2003}
T.~Zhang, Sequential greedy approximation for certain convex optimization
  problems, IEEE Trans. Inform. Theory 49~(3) (2003) 682--691.

\bibitem{DT:2022biorthogonal}
A.~Dereventsov, V.~N. Temlyakov, Biorthogonal greedy algorithms in convex
  optimization, Appl. Comput. Harmon. Anal. 60 (2022) 489--511.

\bibitem{Temlyakov:2015greedyconvex}
V.~Temlyakov, Greedy approximation in convex optimization, Constr. Approx.
  41~(2) (2015) 269--296.

\bibitem{PRK:1993}
Y.~C. Pati, R.~Rezaiifar, P.~S. Krishnaprasad, Orthogonal matching pursuit:
  Recursive function approximation with applications to wavelet decomposition,
  in: Proceedings of 27th Asilomar Conference on Signals, Systems and
  Computers, IEEE, 1993, pp. 40--44.

\bibitem{DT:2019}
A.~Dereventsov, V.~Temlyakov, A unified way of analyzing some greedy
  algorithms, J. Funct. Anal. 277~(12) (2019) 108286.

\bibitem{LS:2024}
Y.~Li, J.~W. Siegel, Entropy-based convergence rates of greedy algorithms,
  Math. Models Methods Appl. Sci. 34~(5) (2024) 779--802.

\bibitem{Blumenson:1960}
L.~E. Blumenson, A derivation of $n$-dimensional spherical coordinates, Amer.
  Math. Monthly 67~(1) (1960) 63--66.

\bibitem{Marsaglia:1972}
G.~Marsaglia, Choosing a point from the surface of a sphere, Ann. Math.
  Statist. 43~(2) (1972) 645 -- 646.

\bibitem{Muller:1959}
M.~E. Muller, A note on a method for generating points uniformly on
  $n$-dimensional spheres, Commun. ACM 2~(4) (1959) 19--20.

\bibitem{Fejes:2022}
G.~Fejes~T{\'o}th, Packing and covering in higher dimensions, arXiv preprint
  arXiv:2202.11358 (2022).

\bibitem{FK:1993}
G.~Fejes~T\'oth, W.~o. Kuperberg, Packing and covering with convex sets, in:
  Handbook of Convex Geometry, {V}ol. {A}, {B}, North-Holland, Amsterdam, 1993,
  pp. 799--860.

\bibitem{rogers1957note}
C.~A. Rogers, A note on coverings, Mathematika 4~(1) (1957) 1--6.

\bibitem{rogers1959lattice}
C.~A. Rogers, Lattice coverings of space, Mathematika 6~(1) (1959) 33--39.

\bibitem{coxeter1959covering}
H.~Coxeter, L.~Few, C.~Rogers, Covering space with equal spheres, Mathematika
  6~(2) (1959) 147--157.

\bibitem{kershner1939number}
R.~Kershner, The number of circles covering a set, Amer. J. Math. 61~(3) (1939)
  665--671.

\bibitem{Bambah:1954}
R.~P. Bambah, Lattice coverings with four-dimensional spheres, Math. Proc.
  Cambridge Philos. Soc. 50~(2) (1954) 203--208.

\bibitem{DR:1963}
B.~N. Delone, S.~S. Ry\v~skov, Solution of the problem on the least dense
  lattice covering of a 4-dimensional space by equal spheres, Dokl. Akad. Nauk
  SSSR 152 (1963) 523--524.

\bibitem{RB:1975}
S.~S. Ry\v{s}kov, E.~P. Baranovski\u{i}, Solution of the problem of the least
  dense lattice covering of five-dimensional space by equal spheres, Dokl.
  Akad. Nauk SSSR 222~(1) (1975) 39--42.

\bibitem{ryshkov1978}
S.~S. Ry\v{s}kov, E.~P. Baranovski\u{i}, {$C$}-types of {$n$}-dimensional
  lattices and {$5$}-dimensional primitive parallelohedra (with application to
  the theory of coverings), Proc. Steklov Inst. Math.~(4) (1978) 140.

\bibitem{CDDN:2020}
A.~Cohen, W.~Dahmen, R.~DeVore, J.~Nichols, Reduced basis greedy selection
  using random training sets, ESAIM Math. Model. Numer. Anal. 54~(5) (2020)
  1509--1524.

\bibitem{Park:2022}
J.~Park, Fast gradient methods for uniformly convex and weakly smooth problems,
  Adv. Comput. Math. 48~(3) (2022) Paper No.~34.

\bibitem{Caflisch:1998}
R.~E. Caflisch, Monte {C}arlo and quasi-{M}onte {C}arlo methods, Acta Numer. 7
  (1998) 1--49.

\bibitem{MC:1995}
W.~J. Morokoff, R.~E. Caflisch, Quasi-{M}onte {C}arlo integration, J. Comput.
  Phys. 122~(2) (1995) 218--230.

\bibitem{CDL:2013}
A.~Cohen, M.~A. Davenport, D.~Leviatan, On the stability and accuracy of least
  squares approximations, Found. Comput. Math. 13 (2013) 819--834.

\bibitem{MNVT:2014}
G.~Migliorati, F.~Nobile, E.~Von~Schwerin, R.~Tempone, Analysis of discrete
  ${L}^2$ projection on polynomial spaces with random evaluations, Found.
  Comput. Math. 14~(3) (2014) 419--456.

\bibitem{CCMNT:2015}
A.~Chkifa, A.~Cohen, G.~Migliorati, F.~Nobile, R.~Tempone, Discrete least
  squares polynomial approximation with random evaluations - application to
  parametric and stochastic elliptic {PDE}s, ESAIM Math. Model. Numer. Anal.
  49~(3) (2015) 815--837.

\bibitem{Hong:2024}
Q.~Hong, J.~Jia, Y.~J. Lee, Z.~Li, Greedy algorithm for neural networks for
  indefinite elliptic problems, arXiv preprint arXiv:2410.19122 (2024).

\bibitem{Nocedal:1999}
J.~Nocedal, S.~J. Wright, Numerical optimization, Springer, 1999.

\bibitem{Fletcher2000}
R.~Fletcher, Practical methods of optimization, John Wiley \& Sons, 2000.

\bibitem{VB:2004convex}
S.~Boyd, L.~Vandenberghe, Convex optimization (2004).

\end{thebibliography}

\appendix

\section{Local optimization using a damped Newton's method}
\label{sec:local-optimization}
In this section, we describe the damped Newton's method with Levenberg-Marquardt regularization \cite{Nocedal:1999, Fletcher2000, VB:2004convex} that we tested as the local optimization algorithm for the argmax subproblem.

At the $n$-th step of a weak greedy algorithm, we want to find $g_n$ such that 
\begin{equation}
    \langle \nabla E(u_{n-1}), g_n \rangle  \geq \gamma_n \mathop{\max}_{g \in \mathbb{D}} \langle \nabla E(u_{n-1}), g \rangle .
\end{equation}
That is, we want to approximately solve $\mathop{\arg\max}_{g \in \mathbb{D}}  \langle \nabla E(u_{n-1}), g \rangle $. 
We state the following equivalent formulations: 
\begin{enumerate}
    \item \begin{equation}
    \omega_n, b_n = \mathop{\arg \max}_{ (\omega, b) \in S^{d-1} \times [c_1,c_2]} |\langle \nabla E(u_{n-1}), \sigma_k(\omega \cdot x + b) \rangle |.
\end{equation}
\item \begin{equation}
    \omega_n, b_n = \mathop{\arg \min}_{ (\omega, b) \in S^{d-1} \times [c_1,c_2]}  \left\{ L(\omega,b) := - \frac{1}{2} \langle \nabla E(u_{n-1}), \sigma_k(\omega \cdot x + b) \rangle^2 \right\} .
\end{equation}
\item \begin{equation}\label{eq:argmax_minimization}
    \phi_n, b_n = \mathop{\arg \min}_{(\phi,b) \in R} \left\{L(\phi,b) = := - \frac{1}{2} \langle \nabla E(u_{n-1}), \sigma_k(\mathcal{S}(\phi) \cdot x + b) \rangle^2  \right\},
\end{equation}
where $\mathcal{S}: [0, \pi]^{d-2} \times [0, 2\pi) \to S^{d-1}$ is the hypersherical coordinate map from a hyprerectangle to a sphere in $\mathbb{R}^d$ given by \eqref{eq:hyperspherical}, and 
$R = [0, \pi]^{d-2} \times [0, 2\pi) \times [c_1,c_2]$ as defined in \eqref{parameter_space}.
\end{enumerate}

We deal with the third formulation \eqref{eq:argmax_minimization}. 
The local optimization algorithm is presented below. 

\begin{algorithm}[H]
\caption{Argmax solver in OGA}
\label{alg:argmax_minimization}
\begin{algorithmic} 
    \State \textbf{Sample:} 
           $\mathbb{D}_N = \{ (\phi^i, b^i) \}_{i=1}^N \sim \text{Uniform}(R)$ 
           \quad \text{(where $R$ is a hyperrectangle as defined in \eqref{parameter_space})}
    \State \textbf{Evaluate:} 
           Compute $L(\phi^i,b^i)$ for all $(\phi^i,b^i) \in \mathbb{D}_N$
    \State \textbf{Screen:} 
           Let $N_1 = \lceil sN \rceil$, where $s \in (0,1)$, and define $\tilde{\mathbb{D}}_{N_1}$ as the $N_1$ pairs with the smallest $L$ values
    \State $\hat{\mathbb{D}} \gets \emptyset$

    \ForAll {$(\phi^0,b^0) \in \tilde{\mathbb{D}}_{N_1}$ } \Comment{In parallel}
        \For{$k=1,\ldots,K$}
            \State $H \gets \bigl[\mathrm{H}_L(\phi^{k-1}, b^{k-1})\bigr] $
            \State $(\delta \phi,\delta b) 
                   \gets 
                   -H^{-1}\,
                    \nabla L(\phi^{k-1}, b^{k-1})$
            \State $\tau \gets 1$
            \State $\lambda =\max \left( \max_i| H_{i,i} | , \epsilon \right)$
            \While{\(
                L(\phi^{k-1} + \tau \delta \phi, b^{k-1} + \tau \delta b) 
                \;>\;
                L(\phi^{k-1}, b^{k-1}) 
                + 
                c \tau \langle \nabla L(\phi^{k-1}, b^{k-1}), (\delta \phi, \delta b)\rangle
            \)}
                \State $\tau \gets \beta \tau$, $\beta \in (0,1)$
                \State $\lambda \gets 2 \lambda $
                \State $(\delta \phi,\delta b) 
                   \gets 
                   -(H + \lambda I)^{-1}\,
                    \nabla L(\phi^{k-1}, b^{k-1})$ \Comment{Levenberg-Marquardt regularization}
            \EndWhile
            \State $(\phi^k, b^k) \gets \bigl(\phi^{k-1} + \tau\,\delta \phi,\; b^{k-1} + \tau\,\delta b\bigr)$
        \EndFor
        \State $\hat{\mathbb{D}} \gets \hat{\mathbb{D}} \,\cup\, \{(\phi^K,b^K)\}$
    \EndFor

    \State \textbf{Final selection:} 
           \[
             (\phi^n,b^n) 
             \gets 
             \mathop{\arg\min}_{(\phi,b) \in \mathbb{D}_N \cup \hat{\mathbb{D}}} 
             L(\phi,b)
           \]
    \State \textbf{Return:} $(\phi^n,b^n)$
\end{algorithmic}
\end{algorithm}
We now discuss an experimental setting to examine the effectiveness of local optimization. 
We solve the $L^2$-minimization problem in 2D. 
We compare the errors to examine whether the use of a local minimization leads to better accuracy. 
We fit the function $u(x_1,x_2) = \sin(\pi x_1)\sin(\pi x_2)$ using a ReLU$^k$ shallow neural network with $k =3$. 
In the randomized OGA, we use a discrete dictionary of size $N = 256$. 
We run the randomized OGA with and without local optimization for 10 trials, respectively, and compute the average $L^2$ error for comparison. 
In the local optimization, we set $s = 0.2$, $\beta = 0.8$, $c = 10^{-4}$. 
For numerical integration, we divide the square domain $[0,1]^2$ into $100^2$ uniform square subdomains and use a Gaussian quadrature of order 3 in each subdomain. 
The comparison result is shown in Table \ref{tab:local_opt_compare}. 
We can observe that the accuracy is improved after using local optimization. 

\begin{table}[H]
	\centering
	\begin{tabular}{c|cc|cc}
		\multicolumn{1}{c}{} & \multicolumn{2}{c|}{\textbf{With Local Optimization}} & \multicolumn{2}{c}{\textbf{Without Local Optimization}} \\ \hline \hline
		\textbf{$n$} & $\|u - u_n\|_{L^2}$ & order & $\|u - u_n\|_{L^2}$ & order \\ \hline
		4     & 1.64e-01 & *    &   1.86e-01 &		 *    \\ \hline
		8     &  8.77e-02 &   0.90   &  9.23e-02 &  		 1.01  \\ \hline
		16   &  3.46e-03 &  		 4.66  &  5.28e-03 &  		 4.13  \\ \hline
		32  &  6.95e-04 &  		 2.32  & 7.85e-04 &  		 2.75 \\ \hline
		64   &  1.11e-04 &  		 2.65  & 1.28e-04 &  		 2.62  \\ \hline
		128  &  1.86e-05 &  		 2.58  &  2.14e-05 &  		 2.58   \\ \hline
		256 &  3.03e-06 &  		2.61  & 3.25e-06 &  		 2.72  \\ \hline 
	\end{tabular}
	\caption{Comparison of the $L^2$ errors $\|u - u_n\|_{L^2}$ and convergence order for methods with and without using the local optimization in the argmax subproblem. The reported errors are averaged over 10 trials.}
	\label{tab:local_opt_compare}
\end{table}

\section{Randomized OGA for solving the general elliptic PDE}
\label{sec:oga-general-elliptic}

In this section, we describe how we can modify the randomized OGA slightly to obtain an algorithm for solving the general elliptic PDE. 

\begin{algorithm}[H]
\caption{OGA with randomized discrete dictionaries}
\label{alg:OGA-randomized-general-elliptic}
\begin{algorithmic}
    \State Given $u_0 = 0$,

    \For{$n = 1, 2, \dots$}
        \State Sample $(\phi_{i,n}, b_{i,n}) \sim \mathrm{Uniform}(R)$ for $1 \leq i \leq N$, to generate a randomized discrete dictionary
        \begin{equation*}
            \mathbb{D}_{N, n} = \left\{ \sigma_k ( \mathcal{S} (\phi_{i, n}) \cdot x + b_{i, n} ) : 1 \leq i \leq N \right\}.
        \end{equation*}
        \State $\displaystyle
        g_n = \operatornamewithlimits{\arg\max}_{g \in \mathbb{D}_{N,n}} \left| \langle \nabla g, \nabla u_{n-1}\rangle  + \langle g, \textbf{b}(x) \cdot \nabla u_{n-1}(x) + c(x) u_{n-1}(x) -f(x)\rangle \right|.$
        \State  Solve the variational problem: Find $u_n \in H_n :=\operatorname{span} \{ g_1, \dots, g_n \}$ such that 
        \begin{equation}
            \langle \nabla u_n , \nabla g_i \rangle + \langle \textbf{b}(x) \cdot \nabla u_{n}, g_i \rangle + \langle 
             c u, g_i\rangle = \langle f, g_i \rangle ~~ \text{for all~} g_i, i = 1,2,...,n. 
        \end{equation}
    \EndFor
\end{algorithmic}
\end{algorithm}

\end{document}